\documentclass[a4paper,11pt]{article}
\usepackage{authblk}

\usepackage[a4paper, left=1in, right=1in, top=1in, bottom=1in]{geometry}

\usepackage[skip=0.5em]{parskip}
\usepackage{microtype}

\usepackage[sort,nocompress]{cite}

\usepackage[title]{appendix}%
\usepackage{textcomp}%
\usepackage{manyfoot}%
\usepackage{listings}%
\usepackage{amsmath,amsthm,verbatim,amssymb,amsfonts,amscd,graphicx}
\usepackage{xcolor}
\usepackage{float}
\usepackage{bbm}
\usepackage{booktabs}
\usepackage{placeins}
\usepackage{thmtools}
\usepackage{thm-restate}
\usepackage{array}
\usepackage{makecell}
\usepackage{enumitem}
\usepackage{mathtools}
\usepackage{xspace}
\usepackage{tikz-cd}
\usepackage{tabularx}
\usepackage{multirow}
\usepackage{multicol}

\usepackage{hyperref}
\usepackage[noabbrev, capitalize]{cleveref}

\newcommand{\Set}{\mathrm{Set}}
\newcommand{\Gph}{\mathrm{Gph}}
\newcommand{\Id}{\mathbbm{1}}
\newcommand{\tensor}{\otimes}
\newcommand{\bigtensor}{\bigotimes}
\newcommand{\atensor}{\oplus}

\DeclareMathOperator{\diag}{diag}
\DeclareMathOperator{\supa}{sup}

\newcounter{dummy}\numberwithin{dummy}{section}
\newtheorem{theorem}[dummy]{Theorem}

\newtheorem{lemma}[dummy]{Lemma}
\newtheorem{definition}[dummy]{Definition}
\newtheorem{proposition}[dummy]{Proposition}
\theoremstyle{remark}

\newtheorem{example}[dummy]{Example}

\newenvironment{cleverproof}[1]
  {\begin{proof}[Proof of \Cref{#1}]}
  {\end{proof}}

\title{Monoidal Rips: Stable Multiparameter Filtrations of Directed Networks}

\author[1]{Nello Blaser}
\author[2]{Morten Brun}
\author[1]{Odin Hoff Gardaa}
\author[1,2]{Lars M. Salbu}

\affil[1]{Department of Informatics, University of Bergen, Norway}
\affil[2]{Department of Mathematics, University of Bergen, Norway}
\affil[ ]{\textit{\{nello.blaser,morten.brun,odin.garda,lars.salbu\}@uib.no}}

\date{}

\begin{document}

\maketitle
\begin{abstract}
    We introduce the monoidal Rips filtration, a filtered simplicial set for weighted directed graphs and other lattice-valued networks.
Our construction generalizes the Vietoris-Rips filtration for metric spaces by replacing the maximum operator, determining the filtration values, with a more general monoidal product.
We establish interleaving guarantees for the monoidal Rips persistent homology, capturing existing stability results for real-valued networks.
When the lattice is a product of totally ordered sets, we are in the setting of multiparameter persistence. Here, the interleaving distance is bounded in terms of a generalized network distance. We use this to prove a novel stability result for the sublevel Rips bifiltration.
Our experimental results show that our method performs better than flagser in a graph regression task, and that combining different monoidal products in point cloud classification can improve performance.
\end{abstract}

\section{Introduction}
Many methods in topological data analysis (TDA) are based on persistent homology and focus on deriving topological descriptors from point cloud data. Common examples include filtrations derived from the \v{C}ech complex, the Vietoris-Rips complex and the alpha complex (see  \cite{chazal2021introduction} for a comprehensive overview). A point cloud in a metric space can be represented as a complete weighted graph where the nodes are the points and the edge weights are determined by the metric. In this setting, the graph is symmetric and satisfies the triangle inequality. Other types of data can naturally be represented as (weighted directed) graphs that are not necessarily symmetric or satisfy the triangle inequality. Examples include social networks, citation networks, and biological networks. A (real-valued) network is a function $G\colon V\times V\to P$ with $P$ being a subset of the extended reals, typically $\mathbb{R}$ or $[0,\infty]$. Methods for computing persistent homology of real-valued networks include directed flag complexes \cite{reimann2017cliques, lutgehetmann2020computing, caputi2021promises}, the directed Rips complex \cite{turner2019rips,mendez2023directed} and the Rips and Dowker filtrations for networks \cite{chowdhury2018functorial}. Another direction is path homology, first introduced in \cite{grigor2013path}. Path homology has been extended to persistent path homology \cite{chowdhury2018persistent}, including grounded and weighted versions \cite{chaplin2024grounded,lin2019weighted}. Computational aspects of persistent path homology are discussed in \cite{dey2022computational,dey2022efficient}. For a review of homology-based methods for networks and their applications, see \cite{aktas2019persistence}.

When working with point cloud data, simplicial complexes are typically used. While simplicial complexes are fundamental in TDA, their simplices---being sets---lack an ordering of vertices. Hence, they are not a natural choice for modeling directed structures, even though such directionality can in principle be encoded indirectly via higher-dimensional simplices. Directed simplicial complexes, which are used for the directed flag complex \cite{reimann2017cliques}, address this by introducing an ordering of the vertices in a simplex. They still disallow tuples with repeated vertices, which limits their expressiveness. For example, the cycle $w = (u, v) + (v, u)$ cannot be the boundary of any $2$-chain, causing it to persist indefinitely. In contrast, simplicial sets allow repeated vertices, enabling $w$ to be the boundary of the $2$-simplex $(u, v, u)$. This greater flexibility makes simplicial sets more suitable for analyzing directed structures. The $\ell_p$-Vietoris-Rips construction \cite{ivanov2024ell_p}, Cho's nerve construction \cite{cho} and the directed Rips complex \cite{turner2019rips} are simplicial sets.

Stability guarantees are an important part of the  theoretical appeal of TDA because they guarantee robustness to noise. The Gromov-Hausdorff distance is a popular choice for measuring the distance between compact metric spaces, including point clouds. A fundamental stability result in TDA tells us that the interleaving distance between \v{C}ech (and Vietoris-Rips) persistent homology modules built from two point clouds is bounded in terms of the Gromov-Hausdorff distance between them. See, for example, \cite[Section~5.7]{chazal2021introduction}. The network distance \cite{carlsson2013axiomatic,carlsson2014hierarchical,chowdhury2016distances,chowdhury2023distances} generalizes the Gromov-Hausdorff distance to real-valued networks. In \cite{chowdhury2023distances}, the authors provide a detailed study of the space of infinite networks, assumed compact and first countable, equipped with the network distance. Their contributions include a characterization of its metric and isomorphism structure, extensions of persistent homology methods to this setting, and convergence results for clustering methods. Many of the previously mentioned methods for real-valued networks have stability guarantees with respect to the network distance \cite{turner2019rips,chowdhury2018functorial,chowdhury2018persistent,mendez2023directed,ivanov2024ell_p}.

\subsection*{Contributions}

In this paper, we do not restrict ourselves to real-valued networks. Rather, we consider $L$-graphs, which are functions $G\colon V\times V\to L$ where $V$ is a finite set and $L$ is a complete lattice. Additionally, we require a commutative monoidal structure $(L,\tensor,e)$ that is compatible with the lattice order in the following sense:
\begin{itemize}
    \item The product is functorial: $s\tensor t\leq s'\tensor t'$ whenever $s\leq s'$ and $t\leq t'$ in $L$, and
    \item the neutral element $e$ is minimal in the sense that $e\leq t$ for all $t$ in $L$.
\end{itemize}
If $L=(L,\leq,\tensor,e)$ satisfies the above conditions, we call $L$ a \emph{symmetric monoidal lattice}. A recurring example of a symmetric monoidal lattice is $[0,\infty]$ equipped with the usual order and the following monoidal structure: For $1\leq p<\infty$, the $p$-sum $+_p\colon [0,\infty]\times [0,\infty]\to [0,\infty]$ is defined as $a+_p b=(a^p+b^p)^{1/p}$, and for $p=\infty$ we define $a+_\infty b=\sup\{a,b\}$.

Given an $L$-graph $G\colon V\times V\to L$, the \emph{monoidal Rips filtration} of $G$ is an $L$-filtered simplicial set $R^\bullet_\tensor(G)$, where the $n$-simplices are the elements in $V^{n+1}$ and their filtration values depend solely on $G$ and the product $\tensor\colon L\times L\to L$. When the monoidal product is the $\infty$-sum on $[0,\infty]$, the monoidal Rips filtration agrees with the directed Rips filtration from \cite{turner2019rips} (\cref{ex:recovering_directed_rips}). Furthermore, we recover network distance stability for the directed Rips filtration (\cref{ex:directed_rips_stability}). The $\ell_p$-Vietoris-Rips filtration from \cite{ivanov2024ell_p} is a filtered simplicial set constructed from a metric space, and is like our construction, inspired by Cho's work in \cite{cho}. We show that the $\ell_p$-Vietoris-Rips simplicial set is a special case of the monoidal Rips filtration when the $L$-graph is a metric, and the monoidal product is the $p$-sum defined above (\cref{ex:recovering_ell_p_vietoris_rips}). We also establish a Gromov-Hausdorff stability result for the $\ell_p$-Vietoris-Rips filtration that is slightly stronger than the one presented in \cite{ivanov2024ell_p} (\cref{ex:ell_p_vietoris_rips_stability}). 

When speaking of stability, we consider \emph{symmetric duoidal lattices} $L=(L,\leq,\tensor,\atensor,e)$. Here we have two symmetric monoidal products $\tensor$ and $\atensor$ on the same lattice $(L,\leq)$ sharing the same neutral element $e$. In addition, these are required to satisfy the following generalized Minkowski inequality:
\begin{equation*}
(s\atensor t)\tensor(s'\atensor t')\leq (s\tensor s')\atensor(t\tensor t')\text{ for all }s,t,s',t'\in L.
\end{equation*}

An example of a symmetric duoidal lattice is $([0,\infty],\leq, +_p,+_q,0)$ where $\leq$ is the usual ordering and $1\leq q\leq p\leq \infty$. One can also construct a symmetric duoidal lattice from any symmetric monoidal lattice $(L,\leq,\atensor,e)$ by defining the product $s\tensor t=\sup\{s,t\}$. In this case, we write $\tensor=\sup$.

A related algebraic structure known as a dioid, which features two monoidal operations, was studied in \cite{segarra2016thesis}. There, they also considers dioid spaces, that is, networks taking values in dioids. While dioids share some similarities with symmetric duoidal lattices, they differ in several key aspects: the order in a dioid is induced from one of the monoidal operations, distributivity is required, and the two operations may have distinct neutral elements, among other differences.

The product $\tensor$ is used to define the monoidal Rips filtration, while $\atensor$ is used to define interleavings and the \emph{(directed) graph distance} between $L$-graphs having the same vertex set. Given $L$-graphs $G$ and $G'$ with vertex set $V$, we define the directed graph distance
\begin{equation*}
\vec{\delta}_\atensor(G,G')=\sup_{v,v'\in V}\inf\left\{c\in L\mid G'(v,v')\leq G(v,v')\atensor c\right\}.
\end{equation*}
The graph distance $\delta_\atensor(G,G')$ is the supremum of $\vec{\delta}_\atensor(G,G')$ and $\vec{\delta}_\atensor(G',G)$. We establish the following interleaving guarantee in terms of the graph distance, denoting the $k$-fold product of $t\in L$ with respect to $\tensor$ by $t^{\tensor k}$:
\begin{restatable*}[Graph Distance Stability]{corollary}{graphdistancestability}
\label{cor:func_graph_distance_stability}
Let $(L,\leq,\tensor,\atensor,e)$ be a symmetric duoidal lattice. For any two $L$-graphs $G,G'\colon V\times V\to L$ and $n\geq 0$, the corresponding monoidal Rips persistence modules $H_n(R_\tensor^\bullet(G))$ and $H_n(R_\tensor^\bullet(G'))$ are $\delta_\atensor(G,G')^{\tensor (n+1)}$-interleaved.
\end{restatable*}

Suppose $G_1$ and $G_2$ are $L$-graphs with possibly distinct vertex sets $V_1$ and $V_2$, respectively. A correspondence between $V_1$ and $V_2$ is a subset $C$ of $V_1\times V_2$ with surjective projection maps $\pi_i\colon C\to V_i$ for $i=1,2$. We consider the induced graphs $\pi_i^*G_i\colon C\times C\to L$ defined by precomposing $G_i$ with $\pi_i$ in both factors. An $L$-graph $G$ is said to have $e$-diagonal if $G(v,v)=e$ for all vertices $v$. Based on the above result, we derive our main stability result for $L$-graphs:
\begin{restatable*}[Correspondence Stability]{theorem}{correspondencestability}
\label{thm:correspondence_stability}
Let $(L,\leq,\tensor,\atensor,e)$ be a symmetric duoidal lattice, and let $G_1\colon V_1\times V_1\to L$ and $G_2\colon V_2\times V_2\to L$ be $L$-graphs. If both $G_1$ and $G_2$ have $e$-diagonal, or $\tensor=\sup$, then for all correspondences $C$ between $V_1$ and $V_2$ with projection maps $\pi_i\colon C\to V_i$, the monoidal Rips persistence modules $H_n(R_\tensor^\bullet(G_1))$ and $H_n(R_\tensor^\bullet(G_2))$ are $\delta_\atensor(\pi_1^*G_1,\pi_2^*G_2)^{\tensor (n+1)}$-interleaved for every $n\geq0$.
\end{restatable*}

Additional assumptions on the symmetric duoidal lattice enable a discussion of multiparameter persistence for graphs. If $L$ is the $m$-fold product of a symmetric monoidal (or duoidal) lattice whose underlying lattice is totally ordered, we can define the \emph{interleaving distance} $d_I^\atensor$ on $L$-persistence modules, and the \emph{generalized network distance} $d_\mathcal{N}^\atensor$ on $L$-graphs. In this multiparameter setting, we establish the following stability result:

\begin{restatable*}[Generalized Network Distance Stability]{theorem}{generalizednetworkdistancestability}\label{thm:gen_network_dist_stability}
Let $G_1\colon V_1\times V_1\to T^m$ and $G_2\colon V_2\times V_2\to T^m$ be graphs where $(T^m,\leq,\tensor,\atensor,e)$ is the $m$-fold product of a symmetric duoidal lattice with $(T,\leq)$ a totally ordered set. If both $G_1$ and $G_2$ have $e$-diagonal, or $\tensor=\sup$, then
\begin{equation*}
d_I^\atensor(H_n(R^\bullet_\tensor(G_1)), H_n(R^\bullet_\tensor(G_2)))\leq d_\mathcal{N}^\atensor(G_1,G_2)^{\tensor (n+1)}
\end{equation*}
for all $n\geq 0$. In particular, for $\tensor=\sup$, we have
\begin{equation*}
d_I^\atensor(H_n(R^\bullet_{\sup}(G_1)), H_n(R^\bullet_{\sup}(G_2)))\leq d_\mathcal{N}^\atensor(G_1,G_2).
\end{equation*}
\end{restatable*}
As a concrete example of multipersistence, we demonstrate how the sublevel Rips bifiltration \cite{botnan2022introduction,alonso2024probabilistic,lesnick2019lecture,lesnick2015theory}, can be written as a monoidal Rips filtration over $[0,\infty]^2$. Furthermore, by applying \cref{thm:gen_network_dist_stability} to the sublevel Rips bifiltration, we bound the interleaving distance in terms of the generalized network distance. We also provide an explicit expression for the generalized network distance in this case.

We have implemented the monoidal Rips filtration for $[0,\infty]$-graphs with $p$-sum as the monoidal product (the $p$-Rips filtration). In our experiments, we combine the $p$-Rips filtration together with persistence images \cite{adams2017} in a machine learning pipeline. The first experiment is a graph regression problem, where we sample random weighted directed graphs using the Directed Random Geometric Graphs (DRGG) model from \cite{michel2029} and  estimate the original DRGG parameter $\alpha$ using LASSO regression. The second experiment concerns point cloud classification, where we obtain point clouds by iterating the linked twist map (a discrete dynamical system) using one of five different parameter values. We then fit and evaluate an ensemble classifier on persistence images following a similar experiment presented in \cite{adams2017}. Our implementation is available at \url{https://github.com/odinhg/Monoidal-Rips}.

\subsection*{Outline}
We begin by establishing the foundational concepts necessary for our discussion in \cref{sec:preliminaries}. Here, we introduce symmetric monoidal- and duoidal lattices, detailing some of their essential properties in \cref{sec:structures_and_bistructures}. Next, we review Galois connections in \cref{sec:galois_connections_and_adjunctions}, which will be useful when we later introduce the graph distance. We then define persistence modules and simplicial sets in \cref{sec:persistent_homology_modules_and_interleaving_distance} and \cref{sec:simplicial_sets}, respectively.

In \cref{sec:graphs_and_the_monoidal_rips_filtration}, we introduce $L$-graphs and the monoidal Rips filtration, and prove functoriality of this construction. We proceed to define the (directed) graph distance in \cref{sec:directed_graph_distance} and show that it has metric-like properties. The discussion in \cref{sec:stability} begins with a proof that surjections induce homotopy equivalences under specific conditions, which is presented in \cref{sec:surjections_induce_homotopy_equivalences}. We then establish our main stability results \cref{cor:func_graph_distance_stability} and \cref{thm:correspondence_stability} in \cref{sec:main_stability_results}.

In \cref{sec:multipersistence}, we work in the multiparameter setting where the underlying lattice is a product of totally ordered sets. Here we define the interleaving distance of $L$-persistence modules, and in \cref{sec:generalized_network_distance}, we define the generalized network distance and establish some of its properties. \Cref{sec:multipersistence_stability} is dedicated to the proof of \cref{thm:gen_network_dist_stability}. We end \cref{sec:multipersistence} with the example of the sublevel Rips bifiltration in \cref{sec:sublevel_rips_bifiltration}. Finally, we outline the implementation of the monoidal Rips filtration in \cref{sec:implementing_p_rips_filtration} and share our experimental results in \cref{subsec:experimental_results}.

\section{Preliminaries}\label{sec:preliminaries}

In this section, we provide the necessary definitions of symmetric monoidal lattice and symmetric duoidal lattice. We also discuss Galois connections within the context of complete lattices. Additionally, we define persistence modules and interleavings between them, along with some fundamental concepts required for working with simplicial sets.

\subsection{Symmetric Monoidal- and Duoidal Lattices}\label{sec:structures_and_bistructures}

In this section, we define two ordered algebraic structures of interest: The symmetric monoidal- and duoidal lattices. Let us briefly motivate the choice of terminology: The name \textit{symmetric monoidal lattice} comes from the more general setting of symmetric monoidal preorders discussed in \cite[Definition~2.2]{fong2018seven}. The term \textit{duoidal} is borrowed from the notion of duoidal categories \cite{nlab:duoidal_category}. In our case, we will assume that the two monoidal products share the same neutral element.

We start by giving some standard definitions of posets, lattices, and monoids. A \textbf{partially ordered set} (or \textbf{poset}) $P=(P,\leq)$ consists of a set $P$ together with a homogenous relation $\leq$ on the set $P$ that is reflexive, antisymmetric, and transitive. We say that a poset $T$ is a \textbf{totally ordered set} if every two elements are comparable, i.e., we either have $t\leq t'$ or $t'\leq t$ for all $t,t'\in T$. Given posets $(P_1,\leq_1),\dots,(P_d,\leq_d)$ we define the \textbf{product order} $\leq$ on $P_1\times\cdots\times P_d$ by $(p_1,\ldots,p_d)\leq (p_1',\ldots,p_d')$ if and only if $p_i\leq p_i'$ for all $i=1,\dots,d$.

Let $P=(P,\leq)$ be a poset and let $S\subseteq P$. An element $l\in P$ is said to be a \textbf{lower bound} of $S$ if $l\leq s$ for all $s\in S$. Similarly, $u\in P$ is an \textbf{upper bound} of $S$ if $s\leq u$ for all $s\in S$. An element $l\in P$ is a \textbf{greatest lower bound} of $S$ if $l$ is a lower bound of $S$ and for every lower bound $l'$ of $S$, we have $l'\leq l$. If a greatest lower bound of $S$ exists, it is necessarily unique, and we refer to it as the \textbf{infimum} $\inf S$ of $S$. Similarly, the \textbf{least upper bound} (or \textbf{supremum}) of $S$, donoted $\sup S$, is the upper bound of $S$ where for every upper bound $u'$ of $S$ we have $\sup S\leq u'$. The poset $P$ is a \textbf{complete lattice} if both $\inf S$ and $\sup S$ exist for every subset $S\subseteq P$.\footnote{The infimum and the supremum are sometimes referred to as the meet and the join, respectively.}

A \textbf{monoid} $M=(M,\tensor,e)$ is a tuple consisting of a (nonempty) set $M$ together with an associative binary product $\tensor\colon M\times M\to M$ and a neutral element $e\in M$. That is, we require
\begin{enumerate}[label=(M\arabic{*}), ref=(M\arabic{*}),leftmargin=5.0em]
    \item\label{cond:m1} $a\tensor(b\tensor c)=(a\tensor b)\tensor c$ and
    \item\label{cond:m2} $a\tensor e=e\tensor a=a$
\end{enumerate}
to hold for all $a,b,c\in M$. A monoid $M$ is \textbf{commutative} if $a\tensor b=b\tensor a$ for all $a,b\in M$. We define the notion of a symmetric monoidal lattice by combining the concepts of complete lattices and commutative monoids as follows:
\begin{definition}
A \textbf{symmetric monoidal lattice} is a tuple $(L, \leq, \tensor, e)$ where $(L,\leq)$ is a complete lattice and $(L,\tensor,e)$ is a commutative monoid such that
\begin{enumerate}[label=(SML\arabic{*}), ref=(SML\arabic{*}),leftmargin=5.0em]
    \item\label{cond:l1} $s\tensor t\leq s'\tensor t'$ whenever $s\leq s'$ and $t\leq t'$ in $(L,\leq)$, and
    \item\label{cond:l2} $e=\inf L$ (or, equivalently, $e\leq t$ for all $t\in L$).
\end{enumerate}
\end{definition}

The condition \ref{cond:l1} indicates that $\tensor\colon L\times L\to L$ behaves as a bifunctor when viewing $L$ as a category. We say that the monoidal product $\tensor$ is \textbf{non-decreasing} if $s\leq s\tensor t$ holds for all $s,t\in L$. In a symmetric monoidal lattice, the product is always non-decreasing. To see this, let $s,t\in L$. By reflexivity and \ref{cond:l2}, we have $s\leq s$ and $e\leq t$, respectively. Applying \ref{cond:m2} and \ref{cond:l1} we get $s=s\tensor e\leq s\tensor t$, and thus we have the following lemma:
\begin{lemma}\label{lem:monoidal_product_is_non_decreasing}
If $(L,\leq,\tensor,e)$ is a symmetric monoidal lattice, then $\tensor$ is non-decreasing.
\end{lemma}

Note that $([0,\infty],\leq,+,0)$ is a symmetric monoidal lattice, whereas $(\mathbb{R},\leq,+,0)$ is not. As a shorthand notation, we denote the $k$-fold product $t\tensor t\tensor\cdots\tensor t$ by $t^{\tensor k}$. Furthermore, for any subset $S=\{s_i\}_{i\in I}\subseteq L$, we use the notation $\bigtensor_{s\in S}s$ or $\bigtensor_{i\in I}s_i$ to represent the product of all the elements $s_i\in S$. 

Later, when we discuss stability, we will consider two possibly distinct monoidal products, $\tensor$ and $\atensor$, on $L$ simultaneously. This allows us to define a filtration in terms of $\tensor$, and on the other hand, distance functions between graphs and persistence modules using $\atensor$.

\begin{definition}   
A \textbf{symmetric duoidal lattice} is a tuple $(L,\leq,\tensor,\atensor, e)$ such that
\begin{enumerate}[label=(SDL\arabic{*}), ref=(SDL\arabic{*}),leftmargin=5.0em]
    \item\label{cond:bl1} both $(L,\leq,\tensor,e)$ and $(L,\leq,\atensor,e)$ are symmetric monoidal lattices, and
    \item\label{cond:bl2} the inequality $(s\atensor t)\tensor(s'\atensor t')\leq (s\tensor s')\atensor(t\tensor t')$ holds for all $s,t,s',t'\in L$.
\end{enumerate}
\end{definition}

The condition \ref{cond:bl2} is a generalization of the Minkowski inequality for $p$-norms as we have
\begin{equation}\label{eq:minkowski_for_p_norm_counting_measure}
    \Vert a+b\Vert_p=\left(\sum_{i=1}^k\vert a_i^p+b_i^p\vert\right)^\frac{1}{p}\leq\left(\sum_{i=1}^k\vert a_i\vert^p\right)^\frac{1}{p}+\left(\sum_{i=1}^k\vert b_i\vert^p\right)^\frac{1}{p}=\Vert a\Vert_p+\Vert b\Vert_p
\end{equation}
for all $a,b\in\mathbb{R}^k$ and $1\leq p\leq\infty$. By induction, \ref{cond:bl2} is equivalent to the following inequality, analogous to \cref{eq:minkowski_for_p_norm_counting_measure}:
\begin{equation*}
    \bigtensor_{i\in I}(s_i\atensor t_i)\leq \left(\bigtensor_{i\in I}s_i\right)\atensor\left(\bigtensor_{i\in I}t_i\right), 
\end{equation*}
for all finite tuples $(s_i)_{i\in I}$ and $(t_i)_{i\in I}$ of $L$. For two symmetric duoidal lattices $(L_1, \leq_1, \tensor_1,\atensor_1, e_1)$ and $(L_2, \leq_2, \tensor_2,\atensor_2, e_2)$, we can construct the \textbf{product symmetric duoidal lattice} on $L_1\times L_2$, equipped with the product order and the coordinatewise monoidal structures. See \cref{sec:productappendix} for more details on products of symmetric monoidal- and duoidal lattices. Furthermore, we can turn any symmetric monoidal lattice $L$ into a symmetric duoidal lattice by considering the supremum $\sup:L\times L\to L$ sending $(s,t)$ to $\sup\{s,t\}$.
\begin{lemma}\label{lemma:supbistructure}
    If $(L,\leq,\atensor,e)$ is a symmetric monoidal lattice, then $(L,\leq,\sup,\atensor,e)$ is a symmetric duoidal lattice.
\end{lemma}
\begin{proof}
    We start by showing that $(L,\leq,\sup,e)$ is a symmetric monoidal lattice. Condition \ref{cond:m1} follows from the fact that $\sup$ is associative. For $t\in L$ we have $t\geq t$ and $t\geq e$, so $t\geq\sup\{t,e\}\geq t$, giving $\sup\{t,e\}=t$ and \ref{cond:m2}. The \ref{cond:l2} condition is independent of the product. For \ref{cond:l1}, consider $s\leq s'$ and $t\leq t'$. Now, $\sup\{s',t'\}$ is an upper bound for $s'$ and $t'$ and thus for $s$ and $t$. Since $\sup\{s,t\}$ is the least upper bound of $s$ and $t$ we have $\sup\{s,t\}\leq \sup\{s',t'\}$. Finally, we show the duoidal inequality \ref{cond:bl2}. Let $x=\sup\{s,s'\}$ and $y=\sup\{t,t'\}$. In particular, we have $s\leq x$ and $t\leq y$, so by \ref{cond:l1} for $\atensor$ we have $s\atensor t\leq x\atensor y$ and similarly $s'\atensor t'\leq x\atensor y$, so $x\atensor y=\sup\{s,s'\} \atensor \sup\{t,t'\}$ gives an upper bound for $s\atensor t$ and $s'\atensor t'$, and hence
    \begin{equation}\label{eq:supminkowski}
        \sup\{s\atensor t, s'\atensor t'\}\leq \sup\{s,s'\} \atensor \sup\{t,t'\}.
    \end{equation}
\end{proof}

We now consider a particular family of symmetric monoidal lattices on $[0,\infty]$ equipped with the usual ordering, where the monoidal product is given by the $\ell_p$-norm on $\mathbb{R}^2$.
\begin{definition}\label{def:p_sum}
For $p\in[1,\infty)$, the \textbf{$p$-sum} is the binary operation $+_p\colon [0,\infty]\times[0,\infty] \to[0,\infty]$ mapping $(a,b)$ to $a+_p b = (a^p+b^p)^{1/p}$. Moreover, we let the \textbf{$\infty$-sum} $a+_\infty b$ of $a$ and $b$ be the limit $\lim_{p\to\infty}a +_p b = \sup\{a,b\}$.
\end{definition}

For all $p\in[1,\infty]$, the tuple  $([0,\infty], \leq, +_p, 0)$ is a symmetric monoidal lattice. Furthermore, from Lemma \ref{lemma:supbistructure}, we get that $([0,\infty], \leq, +_\infty, +_q, 0)$ is a symmetric duoidal lattice, and this can be further refined as follows:

\begin{example}\label{ex:p_q_sum_bistructure}
If $\leq$ is the usual order on $[0,\infty]$ and $1\leq q\leq p \leq \infty$, then $([0,\infty], \leq, +_p, +_q, 0)$ is a symmetric duoidal lattice, i.e., the two monoidal products satisfy \ref{cond:bl2}. When $p,q<\infty$, this follows from \cref{eq:minkowski_for_p_norm_counting_measure} since $p/q\geq1$. Indeed, writing $r=p/q$ we have
\begin{align*}
    &(a+_q b)+_p(c+_q d)=\left((a^q+b^q)^{r}+(c^q+d^q)^{r}\right)^\frac{1}{p}=\left((a^q+b^q)+_r(c^q+d^q)\right)^\frac{r}{p}\\
    \leq&\left((a^q+_r c^q)+(b^q+_r d^q)\right)^\frac{r}{p}=\left((a^p+c^p)^\frac{q}{p}+(b^p+d^p)^\frac{q}{p}\right)^\frac{1}{q}=(a+_p c)+_q(b+_p d).
\end{align*}
If $q<p=\infty$, we write $x=(a^q, c^q)$ and $y=(b^q,d^q)$, and apply the Minkowski inequality as above to get
\begin{align*}
&(a+_q b)+_p(c+_q d)=\sup\{a^q+b^q, c^q+d^q\}^{1/q}=\Vert x+y\Vert_\infty^{1/q}\\\leq&\left(\Vert x\Vert_\infty+\Vert y\Vert_\infty\right)^{1/q}=\left(\sup\{a, c\}^q+\sup\{b, d\}^q\right)^{1/q}=(a+_p c)+_q(b+_p d).
\end{align*}
The last case where $p=q=\infty$ is straightforward:
\begin{align*}
&(a+_q b)+_p(c+_q d)=\sup\{a,b,c,d\}=(a+_p c)+_q(b+_p d).
\end{align*}
\end{example}

\subsection{Galois Connections and Adjunctions}\label{sec:galois_connections_and_adjunctions}
Galois connections are adjunctions between preorders (see \cite[Sec.\ 1.4]{fong2018seven} for a more general exposition). For us, it suffices to consider adjunctions in the setting of endomorphisms of complete lattices. That is, functions $L\to L$ where $L$ is a complete lattice.

\begin{definition}
Let $(L,\leq)$ be a complete lattice. A \textbf{Galois connection} on $L$ is a pair of order-preserving maps $\mu\colon L\to L$ and $\lambda\colon L\to L$ such that
\begin{equation*}
    \lambda(s)\leq t\iff s\leq\mu(t)
\end{equation*}
for all $s,t\in L$. We say that $\mu$ is the \textbf{right adjoint} of $\lambda$, and that $\lambda$ is the \textbf{left adjoint} of $\mu$.
\end{definition}

In general, the adjoint of a function is unique whenever it exists. In a complete lattice $L$, left adjoints always exist, as we can construct the left adjoint $\lambda\colon L\to L$ of any order-preserving function $\mu\colon L\to L$ as follows:
\begin{equation}\label{eq:left_adjoint_as_inf}
    \lambda(s)=\inf\{c\in L\mid s\leq\mu(c)\}.
\end{equation}
Since $\inf A\leq\inf B$ whenever $B\subseteq A\subseteq L$, it follows that $\lambda\colon L\to L$ defined in \cref{eq:left_adjoint_as_inf} is order-preserving. A similar formula for the right adjoint of any order-preserving function can be expressed in terms of the supremum. Thus, in a complete lattice, adjoints always exist. An order-preserving map $\mu\colon L\to L$ is a right adjoint if and only if it preserves infimums \cite[Thm.~1.115]{fong2018seven}, that is, if $\inf\{\mu(s)\mid s\in S\}=\mu(\inf S)$ for all subsets $S\subseteq L$. Consequently, all order-preserving functions $L\to L$ preserve both infimums and supremums. Note that the map $s\mapsto s\tensor s$ is order-preserving. We state the following consequence of the above discussion for later use:
\begin{lemma}\label{lem:inf_product_compatible}
Let $(L,\leq,\tensor,e)$ be a symmetric monoidal lattice. For any subset $S\subseteq L$, we have
\begin{equation*}
    \inf_{s\in S}\{s\tensor s\}=\inf S\tensor \inf S.
\end{equation*}
\end{lemma}

\subsection{Persistence Modules and Interleavings}\label{sec:persistent_homology_modules_and_interleaving_distance}

Persistence modules are commonly indexed over $\mathbb{R}^m$ or a subset thereof. We choose to work in the setting where we replace $\mathbb{R}^m$ and coordinatewise addition with a symmetric monoidal lattice.

For a fixed field $\mathbb{K}$ and a poset $P=(P,\leq)$, a \textbf{$P$-persistence module} (or simply a \textbf{persistence module} when $P$ is fixed) is a functor $H:P\to\operatorname{Vect}_{\mathbb{K}}$. It consists of a collection $\{H_t\}_{t\in P}$ of $\mathbb{K}$-vector spaces together with $\mathbb{K}$-linear \textbf{structure maps} $\{h_{s,t}\colon H_s\to H_t\}_{s\leq t}$ such that $h_{t,t}=\Id_{H_t}$ and $h_{r,t}=h_{s,t}\circ h_{r,s}$ for all $r\leq s\leq t$ in $P$. We allow ourselves to simply write $h$ instead of $h_{s,t}$ if the context is clear.

Let $(L, \leq, \atensor, e)$ be a symmetric monoidal lattice. By an $L$-persistence module, we mean a persistence module over the complete lattice $(L,\leq)$. For $\delta, \epsilon\in L$, a \textbf{$(\delta, \epsilon)$-interleaving} between two $L$-persistence modules $H$ and $H'$ consists of two collections of $\mathbb{K}$-linear maps $F=\{F_t\colon H_t\to H'_{t\atensor\delta}\}_{t\in L}$ and $G=\{G_t\colon H'_t\to H_{t\atensor\epsilon}\}_{t\in L}$ such that the following diagrams commute for all $s\leq t$ in $L$:

\begin{center}
\begin{tikzcd}[column sep=3em]
H_s \arrow[r, "h"] \arrow[d, "F_s"'] & H_t \arrow[d, "F_t"] \\
H'_{s\atensor\delta} \arrow[r, "h'"'] & H'_{t\atensor\delta} 
\end{tikzcd}
\begin{tikzcd}[column sep=3em]
H'_s \arrow[r, "h'"] \arrow[d, "G_s"'] & H'_t \arrow[d, "G_t"] \\
H_{s\atensor\epsilon} \arrow[r, "h"']   & H_{t\atensor\epsilon} 
\end{tikzcd}

\begin{tikzcd}[column sep=1em]
H_t \arrow[rd, "F_t"'] \arrow[rr, "h"] & & H_{t\atensor\delta\atensor\epsilon} \\
 & H'_{t\atensor\delta} \arrow[ru, "G_{t\atensor\delta}"'] & 
\end{tikzcd}
\begin{tikzcd}[column sep=1em]
H'_t \arrow[rd, "G_t"'] \arrow[rr, "h'"] & & H'_{t\atensor\epsilon\atensor\delta} \\
 & H_{t\atensor\epsilon} \arrow[ru, "F_{t\atensor\epsilon}"'] & .
\end{tikzcd}
\end{center}

If such an interleaving exists, we say that $H$ and $H'$ are \textbf{$(\delta, \epsilon)$-interleaved} (with respect to $\atensor$). We refer to a $(\delta,\delta)$-interleaving simply as a $\delta$-interleaving, and say that $H$ and $H'$ are $\delta$-interleaved.

Our notion of interleaving can be viewed as a special case of the general framework introduced in \cite{bubenik2015metrics} and further developed in \cite{de2017theory}. In this more general setting, one considers persistence modules as functors $P\to D$ where $P$ is a preordered set and $D$ is an arbitrary category. Given $\epsilon\in L$, we can consider the order-preserving map $\Omega_\epsilon\colon L\to L$ defined as $\Omega_\epsilon(t)=t\atensor\epsilon$ which satisfies $t\leq\Omega_\epsilon(t)$ for all $t\in L$. In the terminology of \cite{bubenik2015metrics}, an $(\Omega_\delta, \Omega_\epsilon)$-interleaving is precisely what we refer to as a $(\delta,\epsilon)$-interleaving.

The following lemma will be needed for proving \cref{cor:func_graph_distance_stability}:
\begin{lemma}\label{lem:max_interleaved}
    Let $(L, \leq, \atensor, e)$ be a symmetric monoidal lattice. If $H$ and $H'$ are $(\delta, \epsilon)$-interleaved, and $l,m\in L$ with $\delta\leq l$ and $\epsilon\leq m$, then $H$ and $H'$ are $(l,m)$-interleaved. In particular, $H$ and $H'$ are $\sup\{\delta, \epsilon\}$-interleaved.
\end{lemma}
\begin{proof}
For all $t\in L$, we have interleaving maps $F_t\colon H_t\to H'_{t\atensor\delta}$ and $G_t\colon H'_t\to H_{t\atensor\epsilon}$, and we have $t\leq t$. Since $\delta\leq l$ and $\epsilon\leq m$ we get by \ref{cond:l1} for $\atensor$ that $t\atensor\delta\leq t\atensor l$ and $t\atensor\epsilon\leq t\atensor m$. So, we have structure maps $h\colon H_{t\atensor\epsilon}\to H_{t\atensor m}$ and $h'\colon H'_{t\atensor\delta}\to H'_{t\atensor l}$. It can be checked that by defining $F'_t=h'\circ F_t$ and $G'_t=h\circ G_t$ we get an $(l,m)$-interleaving between $H$ and $H'$.
\end{proof}

\subsection{Simplicial Sets}\label{sec:simplicial_sets}
We now define simplicial sets and some related concepts. The \textbf{simplex category} $\Delta$ has objects $[n]=\{0,\dots,n\}$ for integers $n\geq 0$ and order-preserving functions $[m]\to[n]$ as morphisms. For $n\geq0$ and $i\in[n]$, the \textbf{coface map} $d^i\colon [n-1]\to[n]$ is the injective map that hits every element except $i$, and the \textbf{codegeneracy map} $s^i\colon [n+1]\to[n]$ is the surjective map that hits $i$ twice.

\begin{definition}
    A \textbf{simplicial set} $X$ is a functor $X\colon \Delta^{op}\to\Set$, and morphisms of simplicial sets (or \textbf{simplicial maps}) are natural transformations.
\end{definition}

We write $X_n$ for the set of simplices $X([n])$, and for an order-preserving map $\alpha\colon [m]\to[n]$ we write $\alpha^*$ for the function $X(\alpha)\colon X_n\to X_m$. Since the coface and codegeneracy maps generate all morphisms in $\Delta$ (see, for example, \cite[p.~177]{mac2013categories}), we can by functoriality fully describe all functions $X(\alpha)$ by the \textbf{face maps} $d_i=X(d^i)\colon X_n\to X_{n-1}$ and \textbf{degeneracy maps} $s_i=X(s^i)\colon X_n \to X_{n+1}$. A simplex is \textbf{non-degenerate} if it is not in the image of any degeneracy map. A \textbf{simplicial subset} $Y\subseteq X$ is a simplicial set where $Y_n\subseteq X_n$ for all $n\geq 0$ and $Y(\alpha)=X(\alpha)|_{Y_n}$ for all order-preserving functions $\alpha\colon [m]\to[n]$. Given a simplicial set $X$, we let $\operatorname{Sub}(X)$ be the category of simplicial subsets of $X$ with morphisms given by inclusions.

Let $P$ be a poset. A \textbf{$P$-filtration} of a simplicial set $X$ is a functor $P\to\operatorname{Sub}(X)$ from $P$ to the category of simplicial subsets of $X$. Put differently, it is a family of simplicial sets $\{X^t\}_{t\in P}$ such that $X^s\subseteq X^t$ whenever $s\leq t$ in $P$. Given a simplicial set $X$, we can for each dimension $n\geq 0$ define the \textbf{$n$-skeleton} $\operatorname{sk}_n X$, a simplicial subset of $X$ whose $k$-simplices are images of simplices of dimension $n$ or lower. That is,
\begin{equation*}
    \operatorname{sk}_n(X)_k=\{\sigma \in X_k \,|\, \sigma=\alpha^*\tau \textrm{ for some }m\leq n, \tau\in X_m, \alpha\in\Delta([k],[m])\}.
\end{equation*}
Note that $\operatorname{sk}_n(X)_k=X_k$ if $n\geq k$. Furthermore, if $k>n$, then every $k$-simplex in the $n$-skeleton is a degenerate $n$-simplex. In other words, if $\sigma\in\operatorname{sk}_n(X)_k$ for $k>n$, then  $\sigma = S(\tau)$ for some $\tau \in X_n $ where $S\colon X_n\to X_k$ is a composition of degeneracy maps. Our main construction, the monoidal Rips filtration defined in the next section, is a filtration of the singular simplex $EV$ defined as follows:

\begin{definition}
    \sloppy Let $V$ be a finite set. The \textbf{singular $V$-simplex} is the simplicial set $EV$, whose $n$-simplices $y$ are $n+1$-tuples $(v_0,\dots,v_n)$ in $ V^{n+1}$. The simplicial structure is given by the face maps $d_i(y)=(v_0,\dots,\widehat{v_i},\dots,v_n)$ omitting the $i$-th entry, and the degeneracy maps $s_j(y)=(v_0,\dots,v_j,v_j,\dots,v_n)$ duplicating the $j$-th entry. Alternatively, $(EV)_n\cong\operatorname{map}([n],V)$.
\end{definition}
Note that a function $f\colon V\to V'$ induces a simplicial map $f_*\colon EV\to EV'$ defined by sending an $n$-simplex $(v_0,\dots,v_n)$ to $(f(v_0),\dots,f(v_n))$.

Starting with a simplicial complex $(K,V)$, we can construct the simplicial subset $\operatorname{Sing}(K)\subseteq EV$ where
\begin{equation*}
\operatorname{Sing}(K)_n = \left\{(v_0,\ldots,v_n)\mid \{v_0,\ldots, v_n\}\in K\right\}.
\end{equation*}

The simplicial set $\operatorname{Sing}(K)$ has the same topological information as the simplicial complex $K$:
\begin{lemma}[See {\cite[Corollary~6.7]{brun2023dowker} or \cite{camarena}}]\label{lem:geometric_realizations_homotopy_equivalent}
Let $K$ be a simplicial complex. There exists a homotopy equivalence $\varphi\colon|\operatorname{Sing}(K)|\to|K|$ between the geometric realizations of $\operatorname{Sing}(K)$ and $K$, and $\varphi$ is natural in $K$.
\end{lemma}

We use the above lemma in \cref{sec:sublevel_rips_bifiltration} when discussing the sublevel Rips bifiltration in the context of our work.

Similar to simplicial complexes, one can define homology for simplicial sets. For a simplicial set $X$ and a non-negative integer $k$, we denote the $k$-th homology group of $X$ by $H_k(X)$. See \cite[Chapter~1]{goerss2009simplicial} for a detailed exposition on the homology of simplicial sets with coefficients in an Abelian group. Persistent homology of a filtered simplicial set is then defined the usual way where $H_k(X^\bullet)$ is the persistence module mapping $t\mapsto H_k(X^t)$ and the structure maps $h_{s,t}\colon H_k(X^s)\to H_k(X^t)$ for $s\leq t$ are given as the images of the inclusions $X^s\subseteq X^t$ under the functor $H_k$.

\section{\texorpdfstring{$L$}{L}-graphs and the Monoidal Rips Filtration}\label{sec:graphs_and_the_monoidal_rips_filtration}

We now define $L$-graphs for a symmetric monoidal lattice $L$, and introduce the monoidal Rips filtration for such graphs.

\begin{definition}
Let $(L, \leq, \tensor, e)$ be a symmetric monoidal lattice and let $V$ be a finite set. An \textbf{$L$-graph} (or \textbf{graph}) $G$ with vertex set $V$ is a function $G\colon V\times V\to L$. We say that $G$ has \textbf{$e$-diagonal} if $G(v,v)=e$ for all $v\in V$.
\end{definition}
The collection of all $L$-graphs forms a category $\Gph(L)$ with morphisms defined as follows: A morphism $f\colon G_1\to G_2$ between two graphs $G_1\colon V_1\times V_1\to L$ and $G_2\colon V_2\times V_2\to L$ is given by a function $f\colon V_1\to V_2$ on the vertex sets such that $G_2(f(v_1), f(v_1'))\leq G_1(v_1,v_1')$ for all $v_1,v_1'\in V_1$. If $G\colon V\times V\to L$ is a graph and $X$ is a finite set, then functions $f\colon X\to V$ induce graphs $f^*G\colon X\times X\to L$ where $f^*G(x,x')=G(f(x),f(x'))$ for all $x,x'\in X$. By definition, the function $f$ induces a morphism of $L$-graphs
\begin{equation}\label{eq: f_induced_on_f*G}
    f:f^*G\to G.
\end{equation}

Given a graph $G\colon V\times V\to L$ and thinking of $L$ as a constant simplicial set, we define the simplicial map $\mathcal{L}_\tensor^G\colon EV\to L$ assigning to each tuple $(v_0,\ldots, v_n)$ the value
\begin{equation*}
    \mathcal{L}_\tensor^G(v_0,\ldots, v_n) =
    \begin{cases}
        G(v_0,v_0) &\text{if } v_0=v_1=\dots=v_n\text{ and}\\
        \bigtensor_{v_{i-1}\neq v_{i}}G(v_{i-1},v_{i})&  \text{otherwise.}
    \end{cases}
\end{equation*}

The notation $\bigtensor_{v_{i-1}\neq v_{i}}G(v_{i-1},v_{i})$ refers to the product $\bigtensor_{s\in S}s$ where $S$ is the set $S=\{G(v_{i-1},v_i)\mid i\in\{1,\ldots,n\}\text{ and } v_{i-1}\neq v_i\}$. We will use the function $\mathcal{L}_\tensor^G$ to define the filtration value of the simplices in the monoidal Rips filtration. Note that if $G\colon V\times V\to L$ has $e$-diagonal, then the above formula reduces to $\mathcal{L}^G_\tensor(v_0,\ldots,v_n)=\bigtensor_{i=1}^nG(v_{i-1},v_i)$. We consider the functor $R_\tensor\colon \Delta^{op}\times L\times \Gph(L)\to \Set$ sending an object $([n],t,G\colon V\times V\to L)$ to the subset
\begin{equation*}
    R^t_\tensor(G)_n = \left\{y=(v_0,\dots,v_n) \;\middle|\; \mathcal{L}^G_\tensor(\alpha^*y)\leq t \text{ for all }\alpha\in\Delta([m],[n])\right\}\subseteq V^{n+1}.
\end{equation*}

Let us explain how $R_\tensor$ acts on morphisms: If $s\leq t$ in $L$, then we have an inclusion of sets $R^s_\tensor(G)_n\subseteq R^t_\tensor(G)_n$ for all $n$ and $G$. Likewise, an order-preserving map $\alpha\colon [m]\to[n]$, gives a well-defined (natural) function $\alpha^*\colon R^t_\tensor(G)_n\to R^t_\tensor(G)_m$, since $y\in R^t_\tensor(G)_n$ implies that $\mathcal{L}_\tensor^G(\alpha^*y)\leq t$. We define the image of a morphism $f\colon G_1\to G_2$ of $L$-graphs under $R_\tensor$ to be the function $(v_0,\ldots, v_n)\mapsto (f(v_0), \ldots, f(v_n))$.

\begin{proposition}\label{prop:R_tensor_is_a_functor}
For a fixed symmetric monoidal lattice $(L, \leq, \tensor, e)$, we have defined a functor $R_\tensor\colon \Delta^{op}\times L\times \Gph(L)\to \Set$.
\end{proposition}
\begin{proof}
    We discussed functoriality in $\Delta^{op}$ and $L$  above. Functoriality with respect to $\Gph(L)$ remains to be shown. Let $G_1\colon V_1\times V_1\to L$ and $G_2\colon V_2\times V_2\to L$ be graphs. Furthermore, let $f\colon G_1\to G_2$ be a morphism of $L$-graphs, and let $y=(v_0,\dots,v_n)$ be a simplex in $R^t_\tensor(G_1)_n$. If $f(v_0)= f(v_i)$ for all $0\leq i\leq n$, then 
    $
    \mathcal{L}^{G_2}_\tensor(f(y))=G_2(f(v_0),f(v_0))\leq G_1(v_0,v_0)= \mathcal{L}^{G_1}_\tensor(\iota_0^*y)\leq t,
    $
    where $\iota_0\colon [0]\to[n]$ is the map sending $0$ to $0$. If $f(v_0)\neq f(v_i)$, then $v_0\neq v_i$, and by functoriality of $\tensor$ \ref{cond:l1} we have
    \begin{equation*}
        \mathcal{L}^{G_2}_\tensor(f(y))=\bigtensor_{f(v_{i-1})\neq f(v_{i})}G_2(f(v_{i-1}),f(v_{i})) \leq \bigtensor_{f(v_{i-1})\neq f(v_{i})}G_1(v_{i-1},v_{i}).
    \end{equation*}
    Furthermore, since $\tensor$ is non-decreasing we get that
    \begin{equation*}
        \bigtensor_{f(v_{i-1})\neq f(v_{i})}G_1(v_{i-1},v_{i})\leq \bigtensor_{v_{i-1}\neq v_{i}}G_1(v_{i-1},v_{i})=\mathcal{L}^{G_1}_\tensor(y)\leq t.
    \end{equation*}
    Thus, the pointwise map $(v_0,\ldots, v_n)\mapsto (f(v_0), \ldots, f(v_n))$ restricts to a map $\Tilde{f}:R^t_\tensor(G_1)_n\to R^t_\tensor(G_2)_n$. It is clear that this map is the identity if $f=\Id_{G_1}$ and that $\widetilde{f\circ g}=\Tilde{f}\circ \Tilde{g}$.
\end{proof}

Fixing $t\in L$ and an $L$-graph $G\colon V\times V\to L$, we get a simplicial set $R_\tensor^t(G):\Delta^{op}\to\Set$ where $R_\tensor^t(G)_n=R_\tensor([n],t,G)$ and $R_\tensor^t(G)(\alpha)=R_\tensor(\alpha,\Id_t,\Id_G)$ for an order-preserving map $\alpha:[m]\to[n]$. Note that $t\mapsto R_\tensor^t(G)$ is an $L$-filtration of the singular $V$-simplex $EV$.

\begin{definition}[The Monoidal Rips Filtration]\label{def:monoidal_rips_filtration}
Let $L=(L,\leq,\tensor,e)$ be a symmetric monoidal lattice and let $G\colon V\times V\to L$ be an $L$-graph. The collection $R^\bullet_\tensor(G)=\{R^t_\tensor(G)\}_{t\in L}$ defines an $L$-filtered simplicial set called \textbf{the monoidal Rips filtration} of $G$. In the case where $\tensor$ is the $p$-sum from \cref{def:p_sum}, and $\leq$ is the standard order on $[0,\infty]$, we refer to $R^\bullet_{+_p}(G)$ as the \textbf{$p$-Rips filtration}.
\end{definition}

We now give two examples showing how the monoidal Rips filtration generalizes existing constructions.

\begin{example}[The Directed Rips Filtration {\cite{turner2019rips}}]\label{ex:recovering_directed_rips}
An ordered tuple complex $K$ is a collection of tuples that is closed under the removal of vertices. In other words, $K$ is an ordered tuple complex if for all $(v_0,\ldots,v_n)\in K$, we also have $(v_0,\ldots,\hat{v}_i,\ldots,v_n)\in K$ for all $i=0,\ldots,n$. Let $G\colon V\times V\to [0, \infty]$. The \textbf{directed Rips filtration} from \cite{turner2019rips} is the ordered tuple complex defined in filtration degree $t$ as
 \begin{equation*}
     \mathcal{R}^\text{dir}(G)_t = \left\{y=(v_0,\ldots, v_n)\mid G(v_i, v_j)\leq t\text{ for all }i\leq j\right\}.
 \end{equation*}
The condition $G(v_i, v_j)\leq t$ whenever $i\leq j$ is equivalent to $\sup_{i\leq j}G(v_i, v_j)\leq t$, which again is equivalent to requiring that $\mathcal{L}^G_{+_\infty}(\alpha^*y)\leq t$ for all order-preserving maps $\alpha\colon[m]\to[n]$. Thus, when $p=\infty$, the $p$-Rips filtration of $G$ agrees with the directed Rips filtration $\mathcal{R}^\text{dir}(G)$ when forgetting about the degeneracy maps, i.e., as semi-simplicial sets. As discussed in \cite{turner2019rips}, the directed Rips filtration is closed under adjacent repeats and can thus be extended to a filtered simplicial set. So the $\infty$-Rips and the directed Rips filtration are also identical as simplicial sets. We note that the directed Rips filtration can be defined for $\mathbb{R}$-valued graphs as well.
\end{example}

\begin{example}[The $\ell_p$-Vietoris-Rips simplicial set {\cite{ivanov2024ell_p}}]\label{ex:recovering_ell_p_vietoris_rips} Another construcion based on the work of \cite{cho} is the $\ell_p$-Vietoris-Rips simplicial set for metric spaces, which was recently introduced in \cite{ivanov2024ell_p}. Let $(V,d)$ be a metric space, and let $1\leq p\leq\infty$. The \textbf{$\ell_p$-Vietoris-Rips simplicial set} of $V$ in filtration degree $t>0$, denoted $\mathcal{VR}^p_{t}(V)$, has $n$-simplices $y=(v_0,\ldots,v_n)$ such that
\begin{equation*}
    w_p(y):=\max_{0\leq i_0<\dots<i_m\leq n}\Vert d(v_{i_0},v_{i_{1}}),\ldots,d(v_{i_{m-1}},v_{i_{m}})\Vert_p\leq t
\end{equation*}
where the maximum is taken over all subsequences $i_0<\dots<i_m$ of $0,\ldots,n$. When $p<\infty$, we can rewrite the above expression as
\begin{equation*}
w_p(y) = \max_{0\leq i_0<\dots<i_m\leq n}\left(\sum_{k=0}^{m-1} d(v_{i_k},v_{i_{k+1}})^p\right)^{1/p} = \max_{\alpha\in\Delta([m],[n])}\mathcal{L}^d_{+_p}(\alpha^*y).
\end{equation*}
For $p=\infty$, a similar argument holds. Hence, we get the $p$-Rips filtration of the metric $d\colon V\times V\to [0,\infty]$. In other words, $\mathcal{VR}^p_{t}(V)= R_{+_p}^t(d)$.
\end{example}

In \cite[Proposition~2.6]{ivanov2024ell_p}, the authors show that the $\ell_p$-Vietoris-Rips simplicial set agrees with Cho's nerve construction $N_\tensor$ from \cite{cho}. Despite being inspired by Cho's construction, the monoidal Rips filtration is, in general, different as demonstrated by the following example:
\begin{example}\label{ex:not_cho}
    Let $(L,\leq,\tensor,e)$ be a symmetric monoidal lattice, let $G\colon V\times V\to L$ be an $L$-graph, and $t\in L$. Consider the simplicial subset $N^t_\tensor(G)$ of $EV$ from \cite{cho} with $n$-simplices 

    \begin{equation*}
N^t_\tensor(G)_n=\left\{(v_0,\dots,v_n)\,\middle\vert\,\begin{array}{l}
\text{there exist }r_1,\ldots,r_n\text{ in }L\text{ with }\bigtensor_i r_i\leq t\\
\textrm{and } G(v_i,v_j) \leq \bigtensor_{i+1\leq k\leq j} r_k  \textrm{ for all } i\leq j
\end{array}\right\}.
    \end{equation*}

    It is straightforward to show that $N_\tensor^t(G)$ is a simplicial subset of $R_\tensor^t(G)$ for all $t\in L$. However, the converse is not true:
    Consider the $2$-sum $+_2$ and the $[0,\infty]$-graph $G$ defined as
    \begin{equation*}
        \begin{tikzcd}
            v_0\arrow[r,"1"]\arrow[rr,bend left = 30,"3"]\arrow[rrr,bend right = 40, "4"] & v_1\arrow[r,"3"]\arrow[rr,bend right = 20,swap, "3"] & v_2\arrow[r,"1"] & v_3
        \end{tikzcd}
    \end{equation*}
    with $0$-diagonal and where missing edges have value $\infty$. Now, $y=(v_0,v_1,v_2,v_3)$ is a simplex in $R_{+_2}^4(G)$. Assume that $y$ is in $N_{+_2}^4(G)$. This means that there exists $r_1$, $r_2$ and $r_3$, with $i)$ $\sqrt{r_1^2+r_2^2+r_3^2}\leq 4$ and 
    \begin{align*}
        & a)\quad 4=G(v_0,v_3)\leq \sqrt{r_1^2+r_2^2+r_3^2}, & d) \quad 1=G(v_0,v_1)\leq r_1,\\
        & b)\quad 3=G(v_0,v_2)\leq \sqrt{r_1^2+r_2^2}, & e)\quad 3=G(v_1,v_2)\leq r_2, \\
        & c)\quad 3=G(v_1,v_3)\leq \sqrt{r_2^2+r_3^2}, & f)\quad 1=G(v_2,v_3)\leq r_3.
    \end{align*}
    In particular, from $i)$ and $a)$ we have $r_1^2+r_2^2+r_3^2=16$. From $b)$ we have $r_1^2+r_2^2\leq 9$, and so $r_3^2\geq 7$. Similarly, from $c)$ we get $r_1^2\geq 7$. Thus, $16 = r_1^2+r_2^2+r_3^2\geq r_2^2+14$ and $r^2_2\leq 2$, which contradicts $e)$. Hence, $y$ cannot be a simplex in $N_{+_2}^4(G)$.
\end{example}

\section{The Graph Distance}\label{sec:directed_graph_distance}
In this section, we introduce the (directed) graph distance between $L$-graphs on a fixed vertex set. This enables us to define the distortion of a correspondence with respect to two $L$-graphs. Later, in \cref{sec:multipersistence}, where we discuss multipersistence for graphs, we will introduce the generalized network distance by minimizing the graph distance over all correspondences.

Let $(L,\leq,\atensor,e)$ be some fixed symmetric monoidal lattice. For $t\in L$, let $\mu_t\colon L\to L$ be the order-preserving map $c\mapsto t\atensor c$, and let $\lambda_t\colon L\to L$ be its left adjoint. Using \cref{eq:left_adjoint_as_inf}, we get the following expression for $\lambda_t$:
\begin{equation}\label{eq:left_adjoint_lambda_t}
    \lambda_t(s)=\inf\{c\in L\mid s\leq t\atensor c\}.
\end{equation}

\begin{example}\label{ex:left_adjoint_q_sum}
If $([0,\infty],\leq,+_q,0)$ with $q\geq1$, the left adjoint $\lambda_t\colon [0,\infty]\to[0,\infty]$ defined in \cref{eq:left_adjoint_lambda_t} above is given by $\lambda_t(s)=\max\{0,s^q-t^q\}^{1/q}$.
\end{example}

\begin{definition}[Graph Distance]
Given a symmetric monoidal lattice $(L,\leq,\atensor,e)$ and two $L$-graphs $G,G'\colon V\times V\to L$, we define the \textbf{directed graph distance} from $G$ to $G'$ as
\begin{equation*}
    \vec{\delta}_\atensor(G,G')=\sup_{v,v'\in V}\inf\{c\in L\mid G'(v,v')\leq G(v,v')\atensor c\}=\sup_{v,v'\in V}\lambda_{G(v,v')}\left(G'(v,v')\right).
\end{equation*}
The (undirected) \textbf{graph distance} $\delta_\atensor(G,G')$ between $G$ and $G'$ is defined by symmetrizing as follows:
\begin{equation*}
\delta_\atensor(G,G')=\sup\left\{\vec{\delta}_\atensor(G,G'),\,\vec{\delta}_\atensor(G',G)\right\}.
\end{equation*}
\end{definition}
Note that the set $\{c\in L\mid G'(v,v')\leq G(v,v')\atensor c\}$ is nonempty since it at least contains $G'(v,v')$. 
The directed graph distance $\vec{\delta}_\atensor(G,G')$ is defined as a supremum, so we have $\lambda_{G(v,v')}\left(G'(v,v')\right)\leq\vec{\delta}_\atensor(G,G')$ for all $v,v'\in V$. Since $\mu_t$ is the right adjoint of $\lambda_t$, we then get
\begin{equation*}
G'(v,v')\leq\mu_{G(v,v')}\left(\vec{\delta}_\atensor(G,G')\right)=G(v,v')\atensor\vec{\delta}_\atensor(G,G')
\end{equation*}
for all $v,v'\in V$. We summarize this property of the (directed) graph distance in the following lemma:
\begin{lemma}\label{lem:graph_distance_is_a_bound}
    If $(L,\leq,\atensor,e)$ is a symmetric monoidal lattice and $G,G'\colon V\times V\to L$ are $L$-graphs, then the inequality $G'(v,v')\leq G(v,v')\atensor\vec{\delta}_{\atensor}(G,G')$ holds for all $v,v'\in V$.
\end{lemma}

The graph distance is not a metric, as it takes values in a general set $L$. However, in the following lemma, we show that it has metric-like properties:
\begin{lemma}\label{lem:graph_distance_metric}
Fix a finite set $V$ and a symmetric monoidal lattice $(L,\leq,\atensor,e)$, and let $\mathcal{G}(V,L)$ denote the collection of all $L$-graphs with vertex set $V$. If $G,G',G''\in\mathcal{G}(V,L)$, then:
\begin{enumerate}
    \item $\delta_\atensor(G,G')=e$ if and only if $G=G'$,
    \item $\delta_\atensor(G,G')=\delta_\atensor(G',G)$, and
    \item $\delta_\atensor(G,G')\leq \delta_\atensor(G,G'')\atensor\delta_\atensor(G'',G')$.
\end{enumerate}
If $L=([0,\infty],\leq,+_p,0)$, then $(\mathcal{G}(V,L), \delta_{+_p})$ is an extended metric space.
\end{lemma}
\begin{proof}
See \cref{proof:graph_distance_metric}.
\end{proof}

To define distances between graphs with different vertex sets, we use correspondences.

\begin{definition}
A \textbf{correspondence} $C$ between sets $V_1$ and $V_2$ is a subset of $V_1\times V_2$ such that the projection maps $\pi_i\colon C\to V_i$ are surjective for $i=1,2$. We denote by $\mathcal{C}(V_1,V_2)$ the collection of all correspondences between $V_1$ and $V_2$.
\end{definition}

We now state the definition of the distortion of a correspondence, and define the well-known network distance for real-valued networks. Consider functions $G_1\colon V_1\times V_1\to [0,\infty]$ and $G_2\colon V_2\times V_2\to [0,\infty]$. Suppose $C$ is a correspondence between the vertex sets with projection maps $\pi_i\colon C\to V_i$ for $i=1,2$. The \textbf{distortion} of $C$ with respect to $G_1$ and $G_2$ is then defined as (see, e.g., \cite{chowdhury2018functorial})
\begin{equation*}
    \operatorname{dis}(C,G_1,G_2) = \sup_{(v_1,v_2),(v_1',v_2')\in C}\vert G_1(v_1,v_1') - G_2(v_2, v_2')\vert.
\end{equation*}

The \textbf{network distance} (or \textbf{correspondence distortion distance}) $\operatorname{d}_{\mathcal{N}}$ between $G_1$ and $G_2$ is then defined in terms of the smallest distortion taken over all correspondences
\begin{equation*}
    \operatorname{d}_{\mathcal{N}}(G_1, G_2) = \frac{1}{2}\inf_{C\in\mathcal{C}(V_1, V_2)}\operatorname{dis}(C,G_1,G_2).
\end{equation*}

The network distance is a pseudometric for real-valued networks, and it is commonly used to state stability results for topological methods on networks \cite{chowdhury2018functorial, chowdhury2018persistent, mendez2023directed, turner2019rips, carlsson2014hierarchical}. For (compact) metric spaces $d_1:V_1\times V_1\to [0,\infty]$ and $d_2:V_2\times V_2\to [0,\infty]$, the network distance $\operatorname{d}_\mathcal{N}(d_1,d_2)$ agrees with the Gromov-Hausdorff distance $\operatorname{d}_{\textrm{GH}}(V_1,V_2)$ \cite[Theorem~7.3.25]{burago2022course}. 

Note that the set $\mathcal{C}(V_1,V_2)$ is finite as our graphs are assumed to have finite vertex sets. Since we are taking the infimum of a finite subset of $[0,\infty]$, there exists some correspondence such that this infimum is achieved. In the next example, we show that the graph distance agrees with the usual distortion for $[0,\infty]$-graphs.

\begin{example}[Distortion as graph distance]\label{ex:correspondence_distortion}
Consider the symmetric monoidal lattice $([0,\infty], \leq, +, 0)$ where $+$ is addition of real numbers and $\leq$ is the usual ordering. Let $G_1\colon V_1\times V_1\to [0,\infty]$ and $G_2\colon V_2\times V_2\to [0,\infty]$ be graphs, and let $C\subseteq V_1\times V_2$ be a correspondence with projection maps $\pi_i\colon C\to V_i$. The directed graph distance from $\pi_1^*G_1$ to $\pi_2^*G_2$ is now
\begin{equation*}
\vec{\delta}_+(\pi_1^*G_1,\pi_2^*G_2)=\sup_{c,c'\in C}\max\left\{0,G_2(\pi_2(c),\pi_2(c'))-G_1(\pi_1(c),\pi_1(c'))\right\},
\end{equation*}
and similarly 
\begin{equation*}
\vec{\delta}_+(\pi_2^*G_2,\pi_1^*G_1)=\sup_{c,c'\in C}\max\left\{0,G_1(\pi_1(c),\pi_1(c'))-G_2(\pi_2(c),\pi_2(c'))\right\},
\end{equation*}
Symmetrizing gets us the following expression for the graph distance:
\begin{equation*}
\delta_+(\pi_1^*G_1,\pi_2^*G_2) = \sup_{(v_1,v_2),(v_1',v_2')\in C}\vert G_1(v_1,v_1')-G_2(v_2,v_2')\vert=\operatorname{dis}(C,G_1,G_2).
\end{equation*}
\end{example}

\section{Stability}\label{sec:stability}
In this section, we study stability of the Monoidal Rips filtration. We show that surjections on vertex sets induce homotopy equivalences between filtered simplicial sets, a key technical lemma. We then prove stability results for the (directed) graph distance on graphs with the same vertex set, and extend stability to graphs with different vertex sets through correspondences. These results recover known stability results for the directed Rips complex and the $\ell_p$-Vietoris–Rips construction. The results in this section also form the basis for \cref{sec:multipersistence}, where we discuss stability for multi-parameter persistence.

\subsection{Surjections Induce Homotopy Equivalences}\label{sec:surjections_induce_homotopy_equivalences}

We show that graphs induced by surjections give homotopy equivalent monoidal Rips complexes under certain conditions. That is, we prove the following lemma:

\begin{lemma}[Surjections Induce Homotopy Equivalence]\label{lem:surjection_induce_isomorphism_on_homology}
Let $L=(L,\leq,\tensor,e)$ be a symmetric monoidal lattice and let $G\colon V\times V\to L$ be an $L$-graph. If $f\colon X\twoheadrightarrow V$ is surjective with $X$ finite, and either $G$ has $e$-diagonal, or $\tensor=\sup$, then the induced morphism $f:f^*G\to G$ induces homotopy equivalences $R^t_\tensor(f):R^t_\tensor(f^*G)\xrightarrow{\simeq}R^t_\tensor(G)$ for each $t\in L$, sending $n$-simplices $(x_0,\dots x_n)$ to $(f(x_0),\dots, f(x_n))$. This is natural in the sense that if $t\leq t'$ in $L$ then we have a commuting diagram
\begin{center}
    \begin{tikzcd}[row sep=2em, column sep=3em]
    R^t_\tensor(f^*G)\arrow[r,"R^t_\tensor(f)"]\arrow[d,hook] & R^t_\tensor(G)\arrow[d,hook]\\
    R^{t'}_\tensor(f^*G)\arrow[r,"R^{t'}_\tensor(f)"] & R^{t'}_\tensor(G).
    \end{tikzcd}
\end{center}
\end{lemma}
\begin{proof} Naturality follows directly from the fact that $f:f^*G\to G$ is a morphism of $L$-graphs (\cref{eq: f_induced_on_f*G}) and $R^\bullet_\tensor$ is a functor (\cref{prop:R_tensor_is_a_functor}). Furthermore, if $f$ is a bijection, then the statement it is clearly true. So it is enough to show that $R^t_\tensor(f)$ is a homotopy equivalence if $X=V\amalg\{\Tilde{v}\}$ where $f|_V=\Id_V$ and $f(\Tilde{v})=v$ for some $v\in V$. By \cite[Thm I.11.2]{goerss2009simplicial}, we know that the map $f_*\colon R_\tensor^t(f^*G)\to R_\tensor^t(G)$ is a trivial fibration (in particular a homotopy equivalence) if it has the right lifting property with respect to all boundary inclusions. That is, the diagonal map exists for every commutative diagram
    \begin{equation*}
        \begin{tikzcd}
            \partial\Delta^n\arrow[r,"a"]\arrow[d,hook] & R_\tensor^t(f^*G)\arrow[d,"f_*"]\\
            \Delta^n\arrow[r,"b"]\arrow[ru,dashed,"\exists"] & R_\tensor^t(G).
        \end{tikzcd}
    \end{equation*}
    So assume we have such simplicial maps $a$ and $b$. These maps are completely determined by where they map the non-degenerate simplices \cite[Lemma~on~p.177]{mac2013categories}. In $\Delta^n$ the identity map $\Id_n\colon[n]\to[n]$ is the only non-degenerate simplex. In $\partial\Delta^n$ we have the $n+1$ injective coface maps $d^i\colon [n-1]\to [n]$ not hitting $i\in [n]$ as non-degenerate elements. 
    
    Now, if $n=1$, the map $a$ is defined by the vertices $a_1=a(d^0)$ and $a_0=a(d^1)$. By commutativity, we get that $b(\Id_n)=(f(a_0),f(a_1))$ is a simplex in $R^t_\tensor(G)$. In particular, $(a_0,a_1)$ is a simplex in $R^t_\tensor(f^*G)$ since $f^*G(a_0,a_1)=G(f(a_0),f(a_1))=G(b(\Id_1))\leq t$, and we get a lift by sending $\Id_1$ to $(a_0,a_1)$.

    For $n > 1$, let $a(d^i)=(a^i_{0},\dots,\widehat{a^i_{i}},\dots,a^i_{n})$. Since $a$ is a simplicial map we have $d_j(a(d^i))=d_{i}(a(d^{j+1}))$ for $i\leq j$, so 
    \begin{equation*}
        (a^i_{0},\dots,\widehat{a^i_{i}},\dots,\widehat{a^i_{j+1}},\dots,a^i_{n}) = (a^j_{0},\dots,\widehat{a^j_{i}},\dots,\widehat{a^j_{j+1}},\dots,a^j_{n}).
    \end{equation*}
    In particular, $a^i_{l}=a^j_{l}=:a_l$ for all $l\neq i,j+1$. By considering all pairs $i\leq j$ we can write $a(d^i)=(a_0,\dots,\widehat{a_i},\dots,a_n)$ for all $i$. Now, $b(\Id_n)=(f(a_0),\dots, f(a_n))$ is a simplex in $R^t_\tensor(G)$, and we are left to show that $(a_0,\dots,a_n)$ is a simplex in $R^t_\tensor(f^*G)$. By construction, all faces of $(a_0,\dots,a_n)$ are in $R^t_\tensor(f^*G)$. We finish the proof by showing that $\mathcal{L}_\tensor^{f^*G}(a_0,\ldots, a_n)\leq t$. Recall, 
    \begin{equation*}
    \mathcal{L}_\tensor^{f^*G}(a_0,\ldots, a_n) =
    \begin{cases}
        G(f(a_0),f(a_0)) &\text{if } a_0=a_1=\dots=a_n\text{ and}\\
        \bigtensor_{a_{i-1}\neq a_{i}}G(f(a_{i-1}),f(a_{i})) &  \text{otherwise.}
    \end{cases}
    \end{equation*}

    We know that the following value is less than or equal to $t$: 
    \begin{equation*}
    \mathcal{L}_\tensor^{G}(b(\Id_n)) =
    \begin{cases}
        G(f(a_0),f(a_0)) &\text{if } f(a_0)=f(a_1)=\dots=f(a_n)\text{ and}\\
        
        \bigtensor_{f(a_{i-1})\neq f(a_{i})}G(f(a_{i-1}),f(a_{i})) &  \text{otherwise.}
    \end{cases}
    \end{equation*}   
    If $a_0=\dots=a_n$, then $f(a_0)=\dots= f(a_n)$ and the two values are both $G(f(a_0),f(a_0))\leq t$. If $f(a_0)=\dots= f(a_n)$ and $a_{i-1}\neq a_i$ for some $i$, then $a_i\in\{v,\Tilde{v}\}$ for all $i$. In particular, $\mathcal{L}_\tensor^{f^*G}(a_0,\ldots, a_n)$ is a product of $G(f(v),f(\Tilde{v}))=G(f(\Tilde{v}),f(v))=G(f(a_0),f(a_0))$. If $G$ has $e$-diagonal, then they are all $e\leq t$, and so is the product. 
    
    Finally, consider $f(a_{i-1})\neq f(a_i)$ (thus $a_{i-1}\neq a_i$) for some $i$. For $i\in[n]$, if $a_{i-1}\neq a_i$ and $f(a_{i-1})\neq f(a_i)$, then both products get the same contribution $G(f(a_{i-1}), f(a_i))$. Furthermore, if $a_{i-1}\neq a_i$ while $f(a_{i-1})= f(a_i)$, then only the value in $\mathcal{L}_\tensor^{f^*G}(a_0,\ldots, a_n)$ gets a contribution. However, the contribution to the product in this case is $G(f(a_{i-1}),f(a_i))$, which is $e$ when $G$ has $e$-diagonal. Thus,  $\mathcal{L}_\tensor^{f^*G}(a_0,\ldots, a_n) = \mathcal{L}^G_\tensor(b(\Id_n))\leq t$. 
    
    If we have $\tensor=\sup$ without $e$-diagonal, we note that each edge $(a_{i-1},a_{i})$ is a face of $a(d^j)\in R^t_p(f^*G)$ for $j\not\in \{i-1,i\}$. So, $(a_{i-1},a_{i})$ is a simplex in $R^t_p(f^*G)$ and $G(f(a_{i-1}),f(a_{i}))\leq t$. Thus, $t$ gives an upper bound for all elements we take the supremum of in $\mathcal{L}_\tensor^{f^*G}(a_0,\ldots, a_n)$. Since the supremum is the least upper bound we get $\mathcal{L}_\tensor^{f^*G}(a_0,\ldots, a_n)\leq t$. 
\end{proof}
The following example demonstrates that requiring $e$-diagonal is in fact necessary:
\begin{example}
Let $V=\{a\}$, $X=\{a,b\}$ and $f\colon X\to V$ be the unique surjection $f(a)=f(b)=a$. Furthermore, let $\tensor$ be addition of real numbers and consider the graph $G\colon V\times V\to [0,\infty]$ with $G(a,a)=1$. Then $R^\bullet_+(f^*G)$ and $R^\bullet_+(G)$ are not homotopy equivalent. In fact, they do not even have isomorphic persistence modules. It is clear that $H_k(R^t_+(G))$ is trivial for all $t\geq0$ and $k>0$. On the other hand, we see that $H_1(R^t_+(f^*G))$ is non-trivial whenever $t\in[1,2)$ since $w=(a,b)+(b,a)\in C_1(R^1_+(f^*G))$ is a non-trivial $1$-cycle. The only generators in $C_2(R^t_+(f^*G))$ that do not have boundary $(a,a)$ or $(b,b)$ are the simplices $(a,b,a)$ and $(b,a,b)$, both having filtration value $1+1=2$.
\end{example}

\subsection{Main Stability Results}\label{sec:main_stability_results}

In this section, we consider a symmetric duoidal lattice $(L,\leq,\tensor,\atensor,e)$. The monoidal Rips filtration is defined in terms of the product $\tensor$, whereas interleavings and the (directed) graph distance are defined in terms of $\atensor$. Recall that the two products are required to be compatible in the sense of the generalized Minkowski inequality given in \ref{cond:bl2}. We start with the following lemma where we are in the setting of fixed vertex sets:

\begin{lemma}\label{stabilitylemma}
Let $(L,\leq,\tensor,\atensor,e)$ be a symmetric duoidal lattice, and let $G,G'\colon V\times V\to L$ be $L$-graphs. Then, for every $t\in L$, we get an inclusion of $n$-skeletons
\begin{equation*}
\operatorname{sk}_n\left(R_\tensor^t(G)\right)\hookrightarrow\operatorname{sk}_n\left(R_\tensor^{t'}(G')\right)
\end{equation*}
where $t'=t\atensor\vec{\delta}_\atensor(G,G')^{\tensor n}$.
\end{lemma}
\begin{proof}
Let $y=(v_0,\ldots,v_k)$ be a simplex in the $n$-skeleton of $R_\tensor^t(G)$, and let $\alpha\in\Delta([m],[k])$. In particular, we have $\mathcal{L}_\tensor^{G}(\alpha^*y)\leq t$. We now show that $\mathcal{L}_\tensor^{G'}(\alpha^*y)\leq t'$. If all vertices in the simplex $\alpha^*y$ are the same, this follows from \cref{lem:graph_distance_is_a_bound}, condition \ref{cond:l1} and $\tensor$ being non-decreasing:
\begin{align*}
\mathcal{L}_\tensor^{G'}(\alpha^*y)&=G'\left(v_{\alpha(0)}, v_{\alpha(0)}\right)\leq G\left(v_{\alpha(0)}, v_{\alpha(0)}\right)\atensor\vec{\delta}_\atensor(G,G')\\
&=\mathcal{L}_\tensor^{G}(\alpha^*y)\atensor\vec{\delta}_\atensor(G,G')\leq t\atensor\vec{\delta}_\atensor(G,G')\leq t\atensor\vec{\delta}_\atensor(G,G')^{\tensor n}=t'.
\end{align*}
If $\alpha^*y$ is non-degenerate, i.e., there is at least some pair $i\neq j$ such that $v_{\alpha (i)}\neq v_{\alpha (j)}$, we have
\begin{align*}
\mathcal{L}_\tensor^{G'}(\alpha^*y)&=\bigtensor_{\mathclap{v_{\alpha(i-1)}\neq v_{\alpha(i)}}}G'\left(v_{\alpha(i-1)}, v_{\alpha(i)}\right)\leq\bigtensor_{\mathclap{v_{\alpha(i-1)}\neq v_{\alpha(i)}}}\left(G\left(v_{\alpha(i-1)}, v_{\alpha(i)}\right)\atensor\vec{\delta}_\atensor(G,G')\right)\\
&\leq \mathcal{L}_\tensor^{G}(\alpha^*y)\atensor\bigtensor_{\mathclap{v_{\alpha(i-1)}\neq v_{\alpha(i)}}}\vec{\delta}_\atensor(G,G')\leq t\atensor \vec{\delta}_\atensor(G,G')^{\tensor n}=t',
\end{align*}
where the first inequality follows from \cref{lem:graph_distance_is_a_bound} and \ref{cond:l1} for $\tensor$, the second uses the inequality from \ref{cond:bl2},  and the last follows from \ref{cond:l1} for $\atensor$ and the fact that  $v_{\alpha(i-1)}\neq v_{\alpha(i)}$ can happen at most $n$ times.
\end{proof}

Passing to homology, \cref{stabilitylemma} gives us interleaving guarantees in terms of the directed graph distance.

\begin{theorem}[Directed Graph Distance Stability]\label{stability1}
\sloppy Let $(L,\leq,\tensor,\atensor,e)$ be a symmetric duoidal lattice. For any two $L$-graphs $G,G'\colon V\times V\to L$, we have that $H_n(R_\tensor^\bullet(G))$ and $H_n(R_\tensor^\bullet(G'))$ are $(\vec{\delta}_\atensor(G,G')^{\tensor (n+1)},\,\vec{\delta}_\atensor(G',G)^{\tensor (n+1)})$-interleaved.
\end{theorem}
\begin{proof}
    Fix $t\in L$, and write $s=\vec{\delta}_\atensor(G,G')^{\tensor n}$ and $s'=\vec{\delta}_\atensor(G',G)^{\tensor n}$. From \cref{stabilitylemma}, we get compositions of simplicial maps
    \begin{equation*}
        \operatorname{sk}_n(R_\tensor^t(G))\hookrightarrow \operatorname{sk}_n(R_\tensor^{t\atensor s}(G'))\hookrightarrow \operatorname{sk}_n(R_\tensor^{t\atensor s\atensor s'}(G)),
    \end{equation*}
    \begin{equation*}
        \operatorname{sk}_n(R_\tensor^t(G'))\hookrightarrow \operatorname{sk}_n(R_\tensor^{t\atensor s'}(G))\hookrightarrow \operatorname{sk}_n(R_\tensor^{t\atensor s\atensor s'}(G'))
    \end{equation*}
    that are both inclusions. In particular, we get the desired interleaving between the persistence modules $H_m(\operatorname{sk}_n(R_\tensor^\bullet(G)))$ and $H_m(\operatorname{sk}_n(R_\tensor^\bullet(G')))$ for all $m,n\geq 0$. The result follows from the general fact that $H_n(X)=H_n(\operatorname{sk}_{n+1}(X))$ for every simplicial set $X$.
\end{proof}

Going to the undirected graph distance, we get the following corollary of \cref{stability1}:

\graphdistancestability
\begin{proof}
Since $\sup\{s,t\}$ is an upper bound for $s$ and $t$ in $L$, by \ref{cond:l1} we have get that $\sup\{s,t\}^{\tensor n}$ is an upper bound for $s^{\tensor n}$ and $t^{\tensor n}$. In particular, we have $\sup\{s^{\tensor n}, t^{\tensor n}\}\leq\sup\{s,t\}^{\tensor n}$, so
\begin{equation*}
    \sup\left\{\vec{\delta}_\atensor(G,G')^{\tensor n}, \vec{\delta}_\atensor(G',G)^{\tensor n}\right\}\leq\sup\left\{\vec{\delta}_\atensor(G,G'), \vec{\delta}_\atensor(G',G)\right\}^{\tensor n}=\delta_\atensor(G,G')^{\tensor n}.
\end{equation*}
The result now follows from \cref{stability1} and \cref{lem:max_interleaved}.
\end{proof}

Combining \cref{cor:func_graph_distance_stability}, \cref{lem:inf_product_compatible} and \cref{lem:surjection_induce_isomorphism_on_homology}, we consider correspondences between vertex sets and obtain our main stability result. 

\correspondencestability
\begin{proof}
    For every correspondence $C\subseteq V_1\times V_2$ with projection maps $\pi_i\colon C\to V_i$, we have by \cref{cor:func_graph_distance_stability} that $H_n(R_\tensor^\bullet(\pi_1^*G_1))$ and $H_n(R_\tensor^\bullet(\pi_2^*G_2))$ are $\delta_\atensor(\pi_1^*G_1,\pi_2^*G_2)^{\tensor (n+1)}$-interleaved. Since $\pi_i$ is surjective, it follows from \cref{lem:surjection_induce_isomorphism_on_homology} that $R_\tensor^\bullet(\pi_i^*G_i)$ and  $R_\tensor^\bullet(G_i)$ are naturally homotopy equivalent, and hence there is a $\delta_\atensor(\pi_1^*G_1,\pi_2^*G_2)^{\tensor (n+1)}$-interleaving between $H_n(R_\tensor^\bullet(G_1))$ and $H_n(R_\tensor^\bullet(G_2))$.
\end{proof}

We now demonstrate how our stability results specialize to existing stability results for the directed Rips filtration from \cite{turner2019rips} and the $\ell_p$-Vietoris-Rips simplicial set from \cite{ivanov2024ell_p}.

\begin{example}[Directed Rips Stability]\label{ex:directed_rips_stability}
If $\tensor=\sup$, then for $t\in L$, the $n$-fold product $t^{\tensor n}$ is just $t$. Hence, for every correspondence $C\in\mathcal{C}(V_1,V_2)$, we get a $\delta_\atensor(\pi_1^*G_1,\pi_2^*G_2)$-interleaving between $H_n(R^\bullet_{+_\infty}(G_1))$ and $H_n(R^\bullet_{+_\infty}(G_2))$ by \cref{thm:correspondence_stability}. In particular, for the symmetric duoidal lattice $([0,\infty],\leq,+_\infty,+,0)$, 
we recover the stability result of the directed Rips filtration with respect to the network distance \cite[Theorem~21]{turner2019rips}. That is, we have
\begin{equation*}
    d_I(H_n(R^\bullet_{+_\infty}(G_1),H_n(R^\bullet_{+_\infty}(G_2)))\leq 2d_\mathcal{N}(G_1,G_2)
\end{equation*}
for all $n\geq 0$.
\end{example}

\begin{example}[$\ell_p$-Vietoris-Rips Stability]\label{ex:ell_p_vietoris_rips_stability}
Consider the symmetric duoidal lattice $([0,\infty],\leq,+_p,+,0)$ for some $1\leq p\leq\infty$. Note that for $t\in [0,\infty]$, the $n$-fold product $t^{+_pn}=n^{1/p}t$. Thus, given two finite metric spaces $(V_1,d_1)$ and $(V_2,d_2)$, we get by \cref{thm:correspondence_stability} that
\begin{equation*}
d_I(H_n(R_{+_p}^\bullet(d_1)), H_n(R_{+_p}^\bullet(d_2))) \leq 2(n+1)^{1/p}\operatorname{d}_{GH}(V_1,V_2)
\end{equation*}
for all $n\geq 0$ since $\operatorname{d}_{\mathcal{N}}(d_1, d_2)=\operatorname{d}_{GH}(V_1,V_2)$. Hence, we recover the stability restult \cite[Corollary~4.3]{ivanov2024ell_p} for the $\ell_p$-Vietoris-Rips simplicial sets, with a constant factor $2(n+1)^{1/p}$ instead of $2(n+2)^{1/p}$.
\end{example}

\section{Multipersistence and the Generalized Network Distance}\label{sec:multipersistence}

Multiparameter persistence is a generalization of standard (one-dimensional) persistence that allows for the analysis of data that varies across multiple parameters simultaneously. In the multiparameter setting, where one usually considers persistence modules $H\colon\mathbb{R}^m\to\operatorname{Vect}_\mathbb{K}$, the interleaving distance is often defined along the diagonal as follows \cite{botnan2022introduction,bjerkevik2020computing,lesnick2015theory}:
\begin{equation}\label{eq:usual_interleaving_distance_for_multipersistence}
    d_I(H,H')=\inf\{\epsilon\geq0\mid H\text{ and } H'\text{ are }\diag(\epsilon)\text{-interleaved}\}
\end{equation}
where $\diag\colon\mathbb{R}\to\mathbb{R}^m$ denotes the diagonal map $t\mapsto(t,\ldots,t)$. We consider a slight generalization of the interleaving distance in \cref{eq:usual_interleaving_distance_for_multipersistence}. \cite{bubenik2015metrics} and \cite{de2017theory} propose a generalization of the interleaving distance in a different direction, extending it to more general persistence modules via sublinear projections and superlinear families.

We consider $m$-fold products $T^m$ of symmetric duoidal lattices on the form $T=(T,\leq,\tensor,\atensor,e)$ where $(T,\leq)$ is a totally ordered set. On $T^m$, we use the product order and coordinatewise monoidal products (see \cref{sec:productappendix} for more details). Abusing notation, we write $T^m=(T^m,\leq,\tensor,\atensor,e)$ for the rest of this section. The main example to keep in mind is $([0,\infty]^m,\leq,+_p,+_q,0)$ where $1\leq q\leq p\leq\infty$.

We always have the order-preserving maps $\diag\colon T\to T^m$ and $\supa\colon T^m\to T$ given by $t\mapsto(t,\ldots,t)$ and $(t_1,\ldots,t_m)\mapsto\sup_i t_i$, respectively. These maps satisfy $\Id_{T^m}\leq \diag\circ \supa$ and $\supa\circ \diag=\Id_T$.

\begin{definition}
For any two $T^m$-persistence modules $H$ and $H'$, we define the \textbf{interleaving distance} (with respect to $\atensor$) as
\begin{equation*}
    d_I^\atensor(H,H')=\inf\{\epsilon\in T\mid H\text{ and }H'\text{ are }\diag(\epsilon)\text{-interleaved with respect to }\atensor\}
\end{equation*}
where $\diag\colon T\to T^m$ is the map $t\mapsto(t,\ldots,t)$.
\end{definition}

If $T=[0,\infty]$ and $\atensor$ is addition of real numbers, then our interleaving distance $d_I^\atensor(H,H')$ agrees with the usual one in \cref{eq:usual_interleaving_distance_for_multipersistence}.

\subsection{The Generalized Network Distance}\label{sec:generalized_network_distance}

We now introduce a generalized version of the network distance for $T^m$-graphs. Recall that we write $\mathcal{C}(V_1, V_2)$ for the collection of all correspondences between sets $V_1$ and $V_2$.

\begin{definition}[Generalized Network Distance]
For two $T^m$-graphs $G_1\colon V_1\times V_1\to T^m$ and $G_2\colon V_2\times V_2\to T^m$, we define the \textbf{generalized network distance} between $G_1$ and $G_2$ as
\begin{equation}\label{eq:generalized_network_distance}
\operatorname{d}^\atensor_{\mathcal{N}}(G_1, G_2)=\inf_{C\in\mathcal{C}(V_1, V_2)}\supa(\delta_\atensor(\pi_1^*G_1, \pi_2^*G_2))
\end{equation}
where $\pi_i\colon C\to V_i$ are the projection maps, and $\supa\colon T^m\to T$ is the map $(t_1,\ldots,t_m)\mapsto\sup_i t_i$.
\end{definition}

If $T=[0,\infty]$ with the usual order and $m=1$, then $\supa:T\to T$ is the identity map and by \cref{ex:correspondence_distortion} we get that $d^+_\mathcal{N}(G_1,G_1)$ is exactly the network distance $\operatorname{d}_\mathcal{N}(G_1,G_2)$. In particular, for compact metric spaces $(V_1,d_1)$ and $(V_2,d_2)$ we have $\operatorname{d}^+_\mathcal{N}(d_1,d_2)=\operatorname{d}_{\textrm{GH}}(V_1,V_2)$. Similarly, for $1\leq p \leq \infty$ we have that $\operatorname{d}^{+_p}_\mathcal{N}(d_1,d_2)=2^{1/p} \operatorname{d}_{\textrm{GH}}^{(p)}(V_1,V_2)$, where $\operatorname{d}_{\textrm{GH}}^{(p)}$ is the $p$-Gromov-Hausdorff distance from \cite{memoli2021gromov} and \cite{memoli2022metric}.

\begin{proposition}\label{lem:generalized_network_distance_is_psuedo_metric_like}
The generalized network distance satisfies the following properties:
\begin{enumerate}
    \item $\operatorname{d}^\atensor_{\mathcal{N}}(G_1,G_2)=e$ whenever $G_1=G_2$,
    \item $\operatorname{d}^\atensor_{\mathcal{N}}(G_1,G_2)=\operatorname{d}^\atensor_{\mathcal{N}}(G_2,G_1)$, and
    \item $\operatorname{d}^\atensor_{\mathcal{N}}(G_1,G_2)\leq\operatorname{d}^\atensor_{\mathcal{N}}(G_1,G_3)\atensor\operatorname{d}^\atensor_{\mathcal{N}}(G_3,G_2)$
\end{enumerate}
for all $G_1,G_2,G_3\in\Gph(T^m)$. In particular, for the symmetric duoidal lattice $([0,\infty]^m,\leq,+_p,+,0)$ with $1\leq p\leq\infty$ and $m\geq 1$, the generalized network distance defines an extended pseudometric on $\Gph([0,\infty]^m)$.
\end{proposition}
\begin{proof}
See \cref{proof:generalized_network_distance_is_psuedo_metric_like}.
\end{proof}

In contrast to the graph distance, the generalized network distance between two distinct graphs can be $e$ as the following example demonstrates:
\begin{example}
Let $(T,\leq,\tensor,\atensor,e)$ be a symmetric duoidal lattice with $(T,\leq)$ totally ordered. Consider the vertex sets $V_1=\{v_1\}$ and $V_2=\{v_2,v'_2\}$, and define the $T$-graphs $G_i\colon V_i\times V_i\to T$ for $i=1,2$ by letting $G_i(v,v')=e$ for all $v,v'\in V_i$. There is a unique correspondence $C=\{(v_1,v_2),(v_1,v'_2)\}$ between $V_1$ and $V_2$. Writing $\pi_i$ for the projection maps, we have $\pi_1^*G_1=\pi_2^*G_2=G_2$ so $\delta_\atensor(\pi_1^*G_1,\pi_2^*G_2)=e$. Hence, $\operatorname{d}_\mathcal{N}^\atensor(G_1,G_2)=e$ even though $G_1\neq G_2$.
\end{example}

\subsection{Multipersistence Stability}\label{sec:multipersistence_stability}
We now prove stability in the case of $T^m$-graphs by giving an upper bound of the interleaving distance between monoidal Rips persistence modules in terms of the generalized network distance.

\generalizednetworkdistancestability
\begin{proof}
\sloppy Let $C\in\mathcal{C}(V_1,V_2)$ be a correspondence with projection maps $\pi_i\colon C\to V_i$. Write $\epsilon_C$ for the element $\delta_\atensor(\pi_1^*G_1,\pi_2^*G_2)$ in $T^m$. By \cref{thm:correspondence_stability}, the persistence modules $H_n(R^\bullet_\tensor(G_1))$ and $H_n(R^\bullet_\tensor(G_2))$ are $\epsilon_C^{\tensor (n+1)}$-interleaved, and thus also $\diag(\supa(\epsilon_C^{\tensor (n+1)}))$-interleaved by \cref{lem:max_interleaved} since $\Id_{T^m}\leq\diag\circ\supa$. Moreover, it is straightforward to check that $\supa(\epsilon_C^{\tensor (n+1)})=\supa(\epsilon_C)^{\tensor (n+1)}$, using that $\tensor$ is defined pointwise and that $T$ is totally ordered.

Now, let $C_0\in\mathcal{C}(V_1,V_2)$ be a correspondence minimizing the generalized network distance $d_\mathcal{N}^\atensor(G_1,G_2)$. Such a correspondence exists because $T$ is totally ordered and there are only finitely many correspondences between the finite sets $V_1$ and $V_2$. By \cref{lem:inf_product_compatible}, it follows that 
\begin{align*}
d_\mathcal{N}^\atensor(G_1,G_2)^{\tensor (n+1)}&=\left(\inf_{C\in\mathcal{C}(V_1,V_2)} \supa(\epsilon_C)\right)^{\tensor (n+1)}\\&=\inf_{C\in\mathcal{C}(V_1,V_2)} \supa(\epsilon_C)^{\tensor (n+1)}=\supa(\epsilon_{C_0})^{\tensor (n+1)}
\end{align*}
and hence the persistence modules are $\diag(d_\mathcal{N}^\atensor(G_1,G_2)^{\tensor (n+1)})$-interleaved. Since the interleaving distance is defined as an infimum, the result now follows. For the case $\otimes=\sup$, observe that $t^{\tensor (n+1)}=t$ since $\sup\{t,t\}=t$.
\end{proof}

\subsection{Example: The Sublevel Rips Bifiltration}\label{sec:sublevel_rips_bifiltration}

One example of multipersistence is the function Rips bifiltration, which refers to the two-parameter sublevel- and superlevel Rips bifiltrations built from a real-valued function on a finite metric space \cite{botnan2022introduction,alonso2024probabilistic,lesnick2019lecture,lesnick2015theory}. In this section, we demonstrate the flexibility of our approach by showing that the geometric realizations of the sublevel Rips bifiltration and the monoidal Rips filtration of a particular graph are of the same homotopy type. We also adapt \cref{thm:gen_network_dist_stability} to this particular bifiltration. Let us start by defining the usual sublevel Rips bifiltration.

\begin{definition}
Let $(V,d)$ be a finite (extended) metric space and let $\gamma\colon V\to[0,\infty]$ be a function. The \textbf{sublevel Rips bifiltration} $\operatorname{Rips}^\uparrow(\gamma)$ is the filtered simplicial complex defined in filtration degree $(t,s)\in[0,\infty]^2$ as the Rips complex $\operatorname{Rips}\left(\gamma^{-1}[0,s]\right)_t$. 
\end{definition}
A simplex $\sigma\subseteq V$ is in $\operatorname{Rips}^\uparrow(\gamma)_{t,s}$ precisely when $\gamma(v)\leq s$ and $d(v,v')\leq t$ for all $v,v'\in\sigma$.  In applications, the function $\gamma$ can, for example, be some density function on the points in $V$, such as a kernel density estimate. In this case, the filtration is often referred to as the density Rips bifiltration. The sublevel Rips bifiltration can also be defined for functions $\gamma\colon V\to\mathbb{R}$, but since $V$ is finite, one can shift the values of $\gamma$ to $[0,\infty]$ without losing information.

Now, consider the symmetric duoidal lattice $([0,\infty]^2,\leq,\sup,+,0)$ with the usual product order and coordinatewise products. Given a finite (extended) metric space $(V,d)$ and a function $\gamma\colon V\to[0,\infty]$, define the graph $G_\gamma\colon V\times V\to [0,\infty]^2$ by letting $G_\gamma(v,v')=(d(v,v'),\gamma(v))$. 
\begin{proposition}\label{prop:sublevel_rips_equivalent}
    The geometric realizations of $R^{t,s}_\tensor(G_\gamma)$ and $\operatorname{Rips}^\uparrow(\gamma)_{t,s}$ are naturally homotopy equivalent for all $s,t\in[0,\infty]$. In particular, they have isomorphic persistent homology.
\end{proposition}
\begin{proof}
We show that $\operatorname{Sing}(\operatorname{Rips}^\uparrow(\gamma)_{t,s})=R^{t,s}_\tensor(G_\gamma)$ for all $(s,t)\in[0,\infty]^2$. The result then follows from \cref{lem:geometric_realizations_homotopy_equivalent} since the homotopy equivalence $|\operatorname{Sing}(\operatorname{Rips}^\uparrow(\gamma)_{t,s})|\to |\operatorname{Rips}^\uparrow(\gamma)_{t,s}|$ is natural in $\operatorname{Rips}^\uparrow(\gamma)_{t,s}$, and $\operatorname{Rips}^\uparrow(\gamma)$ is a filtered simplicial complex.

First, we consider the degenerate case where all vertices are the same. For $y=(v_0,\ldots,v_0)$, we have $\mathcal{L}^{G_\gamma}_\tensor(\alpha^*y)=G_\gamma(v_0,v_0)=(0,\gamma(v_0))$ for all $\alpha\in\Delta([m],[n])$ and consequently $y\in R^{t,s}_\otimes(G_\gamma)$ if and only if $\gamma(v_0)\leq s$. 

In the non-degenerate case with $y=(v_0,\ldots,v_n)$ and $\alpha\in\Delta([m],[n])$, we have
\begin{equation*}
\mathcal{L}^{G_\gamma}_\tensor(\alpha^*y) = \bigtensor_{v_{\alpha(i-1)}\neq v_{\alpha(i)}}G_\gamma(v_{\alpha(i-1)}, v_{\alpha(i)})= \left(\sup_i d(v_{\alpha(i-1)}, v_{\alpha(i)}),\, \sup_i\gamma(v_{\alpha(i-1)})\right)
\end{equation*}
and hence, using that the metric is symmetric, we get
\begin{equation*}
\max_{\alpha\in\Delta([m],[n])}\mathcal{L}^{G_\gamma}_\tensor(\alpha^*y)=\left(\sup_{i,j\in[n]}d(v_i,v_j),\, \sup_{i\in[n]}\gamma(v_i)\right).
\end{equation*}
So $y\in R^{t,s}_\otimes(G_\gamma)$ if and only if $d(v_i,v_j)\leq t$ and $\gamma(v_i)\leq s$ for all $i,j$.
Thus, $\operatorname{Sing}(\operatorname{Rips}^\uparrow(\gamma)_{t,s})=R^{t,s}_\tensor(G_\gamma)$ and hence $|\operatorname{Rips}^\uparrow(\gamma)_{t,s}|$ and $|R^{t,s}_\tensor(G_\gamma)|$ have the same homotopy type for all $(s,t)\in[0,\infty]^2$ by \cref{lem:geometric_realizations_homotopy_equivalent}. 
\end{proof}

We now demonstrate how our general stability results apply to the sublevel Rips bifiltration. Given $(V_1,d_1,\gamma_1)$ and $(V_2,d_2,\gamma_2)$ where $(V_i,d_i)$ are finite (extended) metric spaces and $\gamma_i\colon V_i\to[0,\infty]$ are functions, we consider the induced graphs $G_{\gamma_i}\colon V_i\times V_i\to [0,\infty]^2$ as above. Let $C\in\mathcal{C}(V_1,V_2)$ be some correspondence between $V_1$ and $V_2$ with projection maps $\pi_1$ and $\pi_2$. Suppose $c,c'\in C$ and write $c=(v_1,v_2)$ and $c'=(v_1',v_2')$. Since $[0,\infty]^2$ is equipped with the product order, we get that
\begin{multline*}
\left\{(t,s)\in [0,\infty]^2\mid\pi_2^*G_{\gamma_2}(c,c')\leq\pi_1^*G_{\gamma_1}(c,c')\atensor (t,s)\right\}\\=\left\{t\in [0,\infty]\mid d_2(v_2,v_2')\leq d_1(v_1,v_1')+t\right\}\times\left\{s\in[0,\infty]\mid\gamma_2(v_2)\leq\gamma_1(v_1)+s\right\}.
\end{multline*}
It follows that the generalized network distance $\operatorname{d}_{\mathcal{N}}^+(G_{\gamma_1},G_{\gamma_2})$ between $G_{\gamma_1}$ and $G_{\gamma_2}$ is given by the following infimum:
\begin{equation*}
\inf_{C\in\mathcal{C}(V_1,V_2)}\max\left\{\sup_{(v_1,v_2),(v_1',v_2')\in C}\vert d_1(v_1,v_1')-d_2(v_2,v_2')\vert,\sup_{(v_1,v_2)\in C}\vert\gamma_1(v_1)-\gamma_2(v_2)\vert\right\},
\end{equation*}
and by \cref{thm:gen_network_dist_stability}, the interleaving distance between the corresponding persistence modules is bounded by this expression. Note that the Gromov-Hausdorff distance between $(V_1,d_1)$ and $(V_2,d_2)$ can be written as 
\begin{equation*}
2d_{GH}(V_1,V_2)=\inf_{C\in\mathcal{C}(V_1,V_2)}\sup_{(v_1,v_2),(v_1',v_2')\in C}\vert d_1(v_1,v_1')-d_2(v_2,v_2')\vert  
\end{equation*}
and the Hausdorff distance between the images $\gamma_1(V_1)$ and $\gamma_2(V_2)$ in $[0,\infty]$ can be written as 
\begin{equation*}
d_H(\gamma_1(V_1),\gamma_2(V_2))=\inf_{C\in\mathcal{C}(V_1,V_2)}\sup_{(v_1,v_2)\in C}\vert\gamma_1(v_1)-\gamma_2(v_2)\vert.    
\end{equation*}
Thus, the generalized network distance in this case is greater or equal to the maximum of $2d_{GH}(V_1,V_2)$ and $d_H(\gamma_1(V_1),\gamma_2(V_2))$. The following example demonstrates that the maximum of $2d_{GH}(V_1,V_2)$ and $d_H(\gamma_1(V_1),\gamma_2(V_2))$ does not give an upper bound for the interleaving distance between sublevel Rips persistence modules:

\begin{example}
Let $V_1=\{a,b\}$ with $d(a,b)=2$, $\gamma_1(a)=2$ and $\gamma_1(b)=0$. Furthermore, let $V_2=\{c,d,e\}$ with $d(c,d)=d(d,e)=3$, $d(c,e)=1$, $\gamma_2(c)=\gamma_2(d)=0$ and $\gamma_2(e)=1$. Considering the $25$ possible correspondences between $V_1$ and $V_2$, one can check that $\operatorname{d}_{\mathcal{N}}^+(G_{\gamma_1},G_{\gamma_2})=2$. On the other hand, both $2d_{GH}(V_1,V_2)$ and $d_H(\gamma_1(V_1),\gamma_2(V_2))$ are $1$. Moreover, it can be shown that for $\epsilon<3/2$, the zero-dimensional persistent homology modules corresponding to the sublevel Rips bifiltrations can not be $\diag(\epsilon)$-interleaved.
\end{example}

\section{Implementation and Experimental Results}\label{sec:implementation_and_experimental_results}

In this section, we first describe how the persistent homology of the $p$-Rips filtrations can be computed, and then present our experimental results.

\subsection{\texorpdfstring{Constructing $[0,\infty]$-graphs From Data}{Constructing [0,infinity]-graphs From Data}}\label{sec:constructing_p_graphs_from_data}

A \textbf{simple directed graph} is a tuple $(V,E)$ consisting of a vertex set $V$ and a set of directed edges $E\subseteq (V\times V)\setminus\Delta_V$ where $\Delta_V=\{(v,v)\,|\, v\in V\}$ is the diagonal of $V$. An \textbf{(edge-) weighted directed graph} is a simple directed graph endowed with a weight function $\omega_E\colon E\to [0,\infty)$. Given a weighted directed graph $(V,E,\omega_E)$, we construct an induced $[0,\infty]$-graph $G\colon V\times V\to[0, \infty]$ by letting
\begin{equation*}
G(v, v') = 
     \begin{cases}
     \omega_E(v,v') &\quad\text{if }v\neq v'\text{ and }(v,v')\in E\text{,}\\
     \infty &\quad\text{if }v\neq v'\text{ and }(v,v')\notin E\text{, and}\\
     0 &\quad\text{if } v=v'.
     \end{cases}
\end{equation*}
If we also have vertex weights $\omega_V\colon V\to [0,\infty)$, we instead let $G(v,v)=\omega_V(v)$. For the case when we have a finite point cloud $X$ in a metric space $(Y, d)$, we consider the $[0,\infty]$-graph $G\colon X\times X\to[0,\infty]$ given by $G=d$.

\subsection{Implementing the \texorpdfstring{$p$}{p}-Rips Filtration}\label{sec:implementing_p_rips_filtration}

Our goal is to compute the persistent homology of the $p$-Rips filtrations. By definition, the monoidal Rips filtration $R_\otimes^t(G)$ includes degenerate simplices with repeated adjacent vertices. For example, simplices on the form $(a,a)$ and $(a,a,b)$. By removing such simplices, we obtain a Delta set (or semi-simplicial set) with naturally isomorphic homology groups by \cite[Theorem~2.1~and~2.4]{goerss2009simplicial}. In particular, they have isomorphic persistent homology modules. Delta sets are more general than simplicial complexes, but less general than simplicial sets, as they come with face maps but lack degeneracy maps (see \cite[Section~2.4]{friedman2008elementary} for a more detailed exposition). By going to Delta sets, we only need to consider non-degenerate simplices. In dimensions $0$, $1$ and $2$, they are on the form $(a)$, $(a,b)$, $(a,b,c)$ and $(a,b,a)$ where $a,b,c\in V$. For $m$ vertices, we end up with at most $m(m-1)^n$ $n$-simplices in our filtration. This results in size complexity $\mathcal{O}(m^{n_{\text{max}}+1})$ for the $p$-Rips filtration where $m$ is the number of vertices and $n_{\text{max}}$ is the maximum homological dimension we are computing.

To check the condition $\mathcal{L}^G_\otimes(\alpha^*y)\leq t$ we can assume $\alpha$ to be injective, and there are finitely ($n!$) many such injective maps. This can be improved by building the filtration iteratively. Let $f\colon R^t_\otimes(G)\to P$ be the map sending a simplex to its filtration value, i.e., $f(y)=\max_{\alpha}\mathcal{L}^G_\otimes(\alpha^*y)$ where $\alpha\in\Delta([m],[n])$. Now, 
\begin{equation*}
f(y) = \max\left(\left\{f(d_i y)\right\}_{i=0}^n\cup\{\mathcal{L}^G_\otimes(y)\} \right).
\end{equation*}
So the filtration value of $y$ depends solely on the filtration values of its faces (already computed in the previous iteration) and $\mathcal{L}^G_\otimes(y)$. We implemented $p$-Rips persistence in Python using PHAT \cite{bauer2017phat} for the persistent homology computations (with integer coefficients modulo $2$).

\subsection{Experimental Results}\label{subsec:experimental_results}
In this section, we present experimental results obtained from computing the $p$-Rips filtration of directed weighted graphs and point clouds. We used Gudhi \cite{gudhi:urm, gudhi:PersistenceRepresentations} to compute persistence images \cite{adams2017}, before applying conventional machine learning methods for regression and classification. For $\infty$-Rips persistence of point clouds, we used Gudhi's implementation of the Vietoris-Rips filtration \cite{gudhi:RipsComplex} to speed up computations, since both filtrations have the same persistent homology for symmetric functions \cite[p.~15]{turner2019rips}. Run times and filtration sizes for the $p$-Rips filtration and Flagser on randomly generated weighted directed graphs are reported in \cref{tab:filtration_runtimes}.

\subsubsection{Geometric Graph Regression}\label{sec:geometric_graph_regression}
The Directed Random Geometric Graph (DRGG) model is a random graph model introduced in \cite{michel2029} that provides directed graphs with several properties found in real-world networks. Given positive integers $n$ and $d$, and a real-valued parameter $\alpha > d+1$, the DRGG model $G(n,\alpha,d)$ generates random directed graphs with $n$ nodes. The generation starts by uniformly sampling $n$ points randomly from the unit cube $[0,1]^d$. Each point corresponds to a vertex in the final graph. A periodic boundary condition is assumed, so we can think of the points as lying on the $d$-dimensional torus $\mathbb{T}^d$. We then assign to each vertex $u$ a radius $r_u$ drawn from a Pareto distribution defined by the probability density function
\begin{equation*}
    f_\alpha(r) = \begin{cases}
        \frac{\eta}{r^\alpha}\quad\text{if }r_0\leq r\leq \frac{1}{2}\text{ and}\\
        0\quad\text{otherwise,}
    \end{cases}
\end{equation*}
where $r_0$ depends only on $n$ and $d$, and $\eta$ is a normalizing factor. The edge $(u,v)$ is added to the graph whenever the distance $d_{\mathbb{T}^d}(u,v)\leq r_v$. In the implementation from \cite{michel2029}, edges are unweighted. We made a slight modification, letting the weight of $(u,v)$ be $d_{\mathbb{T}^d}(u,v)$.

For modelling real-world graphs, it is suggested in \cite{michel2029} to let $1\leq d\leq 5$. Based on this, we fixed $d=4$. We performed $100$ independent experiments with $500$ graphs in each experiment. Each graph was generated from the DRGG model, with the values of $\alpha$ uniformly sampled from the interval $[5,10]$, and $n=200$ fixed. After generating the graphs, they were split into training and test data (in a $80:20$ ratio). Next, we computed the $p$-Rips filtration for each graph, and their persistent homology for different values of $p$. We also performed the same experiments using the software package Flagser \cite{lutgehetmann2020computing}. Flagser computes the persistent homology of the directed flag complex of a directed graph, whose simplices are ordered cliques. It supports custom filtrations where the value of a higher-dimensional simplex is determined from the filtration values of its boundary cells. We use the sum ($p=1$) and maximum ($p=\infty$) functions for this purpose. We choose Flagser as it is easily available, naturally handles weighted directed graphs, and has persistent homology generally distinct from that of the $p$-Rips filtration.

Persistence images were generated for $H_0$, ignoring infinite persistence pairs, as the graphs are almost surely connected \cite{michel2029}, finite persistence pairs in $H_1$, and infinite persistence pairs in $H_1$. The three persistence images of sizes $10$, $10\times 10$ and $10$, respectively, were then concatenated into a single feature vector in $\mathbb{R}^{120}$ representing the given graph. We fitted a standard LASSO regression model on these feature vectors and carried out $5$-fold cross-validation to find the optimal $L_1$-penality term $\alpha_{\text{LASSO}}$ and the Gaussian standard deviation $\sigma$ used for generating the persistence images. Statistics for the mean absolute error (MAE) score computed over the test data for all $100$ experiments are reported in \cref{fig:mae_drgg_flagser_and_prips}. Numerical results, including results for additional values of $p$, can be found in \cref{table:all_mae_scores_drgg_experiment} in \cref{sec:appendix_experimental_results}.\looseness=-1 

\begin{figure}
\centering
\includegraphics{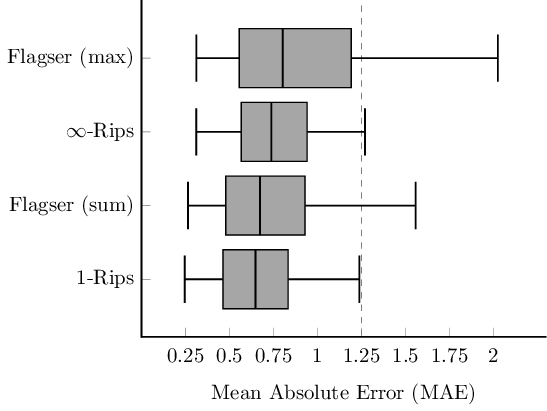}
\caption{Box plot showing mean absolute errors (MAE) for our filtration and Flagser over the $100$ experiments in our geometric graph regression problem. Note that outliers are not shown in the plot. The dashed line $x=1.25$ indicates the expected MAE for the constant function predicting $\hat{\alpha}=7.5$.}
\label{fig:mae_drgg_flagser_and_prips}
\end{figure}

Based on our experiments, the $1$-Rips persistence performed on average slightly better than the corresponding Flagser persistence on this particular task. The most significant difference was seen between using $p=1$ or $p=\infty$, both for $p$-Rips and Flagser.
\FloatBarrier

\subsubsection{Point Cloud Classification}\label{sec:point_cloud_classification}

Despite our method being motivated by graphs, our construction can also be applied to point cloud data, as explained in \cref{sec:constructing_p_graphs_from_data}. In this experiment, we investigate if the $p$-Rips filtration captures information not captured by the Vietoris-Rips filtration. We replicate, to some extent\footnote{Due to time limitations, we used $100$ iterations instead of the $1000$ iterations used in \cite{adams2017}.}, the experiment on parameter value classification in a discrete dynamical model presented in \cite{adams2017}. The \textit{linked twist map} is the discrete dynamical system defined by 
\begin{equation*}
\begin{cases}
x_{n+1} = x_n + ry_n(1-y_n)\\
y_{n+1} = y_n + rx_{n+1}(1-x_{n+1})
\end{cases}
\end{equation*}
computed modulo $1$ for some parameter $r>0$. For an initial point, $(x_0, y_0)$ sampled from the uniform distribution on the unit square, we performed $100$ iterations of the linked twist map to generate a point cloud with $100$ points. For each parameter value $r\in\{2.5,\,3.5,\,4.0,\,4.1,\,4.3\}$, we generated $50$ point clouds. We conducted $100$ independent experiments: Given $250$ point clouds ($50$ for each parameter value) and a fixed $p\geq 1$, we computed persistence diagrams for $H_0$ and $H_1$ using the $p$-Rips filtration with the Euclidean distance. The persistence diagrams were split into a training and test set (in a $70:30$ ratio). We then computed the corresponding persistence images of resolution $20\times 20$ pixels. Maximum birth and persistence values were computed over the training data and used to determine the image ranges. To ensure fair comparison between different values of $p$, we tried different values for $\sigma$, the standard deviation of the Gaussian kernel and picked the one that resulted in the highest accuracy. We concatenated the persistence images for $H_0$ and $H_1$, and used a bagging classifier with a decision tree base estimator to predict the parameter value $r$. We ended up with $\sigma=0.03$ as the optimal standard deviation for all values of $p$ tested.\looseness=-1

We observed a modest increase in accuracy when using $p=1$ compared to $p=\infty$, which might suggest that the filtration using $p<\infty$ captures some information not captured by the usual Vietoris-Rips filtration. To further investigate this hypothesis, we combined the persistence images from the filtrations with $p=1$ and $p=\infty$. We kept everything else the same as in the previous experiment. The optimal value for $\sigma$ found in this case was $0.04$. By combining persistence images, we achieved a mean classification accuracy of $59.32$ (see \cref{fig:accuracy_1_and_inf}). Numerical summary statistics for all values of $p$ we tried are reported in \cref{table:accuracies_point_clouds_different_p} in \cref{sec:appendix_experimental_results}. Based on these findings, we believe it can be beneficial to consider including features from filtrations where $p<\infty$. We note that the accuracy obtained here using $p$-Rips persistence is lower than in previous works \cite{carriere2020perslay, adams2017}, which is expected since our experiments use fewer and smaller point clouds and are designed to study the effect of the parameter $p$ rather than to optimize predictive performance.

\begin{figure}
\centering
\includegraphics{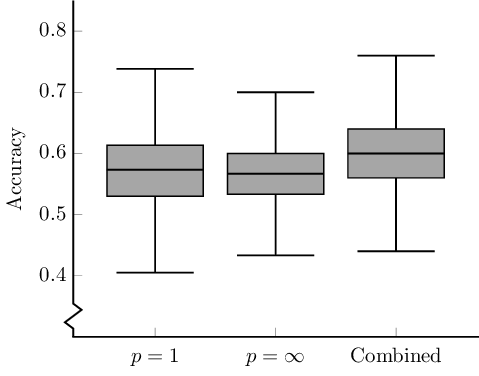}
\caption{Box plot showing accuracies for the different values of $p$ over $100$ experiments in our point cloud classification problem. Note that outliers are not shown in the plot.}
\label{fig:accuracy_1_and_inf}
\end{figure}
\FloatBarrier

\clearpage

\begin{appendices}
\section{Experimental Results and Omitted Proofs}\label{sec:appendix}
\subsection{Experimental Results}
\label{sec:appendix_experimental_results}
We here present the numerical results for the two experiments.
\begin{table}[ht]
\centering
\begin{tabular}{@{}lccccc@{}}
\toprule
\textbf{Filtration} & \textbf{Median} & \textbf{Lower Quartile} & \textbf{Upper Quartile} \\ \midrule
Flagser (max) & 0.803 & 0.555 & 1.192 \\
$\infty$-Rips & 0.738 & 0.567 & 0.941 \\ \midrule
Flagser (sum) & 0.673 & 0.479 & 0.929 \\
$1$-Rips & \textbf{0.647} & \textbf{0.463} & \textbf{0.833} \\ \midrule
$2$-Rips & 0.713 & 0.573 & 0.936 \\
$5$-Rips & 0.777 & 0.535 & 0.976 \\
$10$-Rips & 0.718 & 0.548 & 0.918 \\
$15$-Rips & 0.691 & 0.543 & 0.899 \\
$20$-Rips & 0.710 & 0.554 & 0.922 \\
\bottomrule
\end{tabular}
\caption{Summary of mean absolute errors (MAE) over the $100$ experiments in our geometric \textbf{graph regression} problem for each filtrations.}
\label{table:all_mae_scores_drgg_experiment}
\end{table}

\begin{table}[ht]
\centering
\begin{tabular}{@{}rccccc@{}}
\toprule
\textbf{$p$} & \textbf{Mean} & \textbf{Median} & \textbf{Lower Quartile} & \textbf{Upper Quartile} \\ \midrule
$1$ & 57.19 & 57.33 & 53.00 & 61.33 \\
$2$ & 55.25 & 55.33 & 52.00 & 58.67 \\
$5$ & 56.40 & 57.33 & 52.00 & 60.00 \\
$10$ & 56.53 & 56.67 & 53.33 & 60.00 \\
$15$ & 56.35 & 56.00 & 53.33 & 60.00 \\
$20$ & 56.39 & 56.00 & 53.33 & 60.00 \\
$\infty$ & 56.39 & 56.67 & 53.33 & 60.00 \\
$1$ and $\infty$ & \textbf{59.32} & \textbf{60.00} & \textbf{56.00} & \textbf{64.00} \\
\bottomrule
\end{tabular}
\caption{Mean, median, and the first and third quartile of the accuracies over the $100$ experiments for different values of $p$ in our \textbf{point cloud classification} problem. The last column represents accuracies from combining the persistence images for $1$-Rips and $\infty$-Rips.}
\label{table:accuracies_point_clouds_different_p}
\end{table}

\begin{table}[ht]
\centering
\begin{tabular}{r l l r r}
\toprule
\textbf{Nodes} & \textbf{Method} & \textbf{Filtration} & \textbf{Time [s] (average)} & \textbf{Filtration size (average)} \\
\midrule
\multirow{4}{*}{50}
 & \multirow{2}{*}{Flagser} & max & 0.0067 & \multirow{2}{*}{527.8} \\
 &                           & sum & 0.0049 &                         \\

 & \multirow{2}{*}{$p$-Rips} & max & 0.0235 & \multirow{2}{*}{1510.0} \\
 &                           & sum & 0.0246 &                         \\
\midrule
\multirow{4}{*}{100}

 & \multirow{2}{*}{Flagser} & max & 0.0393 & \multirow{2}{*}{2061.6} \\
 &                           & sum & 0.0367 &                         \\
 
 & \multirow{2}{*}{$p$-Rips} & max & 0.2133 & \multirow{2}{*}{9937.8} \\
 &                           & sum & 0.2176 &                         \\
\midrule
\multirow{4}{*}{150}


 & \multirow{2}{*}{Flagser} & max & 0.2985 & \multirow{2}{*}{4601.0} \\
 &                           & sum & 0.3076 &                         \\
 
 & \multirow{2}{*}{$p$-Rips} & max & 0.9642 & \multirow{2}{*}{31613.4} \\
 &                           & sum & 0.9613 &                         \\
\midrule
\multirow{4}{*}{200}
 
 & \multirow{2}{*}{Flagser} & max & 1.1741 & \multirow{2}{*}{8164.2} \\
 &                           & sum & 0.9929 &                         \\
 

 & \multirow{2}{*}{$p$-Rips} & max & 3.3940 & \multirow{2}{*}{72797.8} \\
 &                           & sum & 3.2444 &                         \\
\midrule
\multirow{4}{*}{250}

 & \multirow{2}{*}{Flagser} & max & 3.9645 & \multirow{2}{*}{12664.8} \\
 &                           & sum & 3.7899 &                         \\
 

 & \multirow{2}{*}{$p$-Rips} & max & 10.5466 & \multirow{2}{*}{137916.4} \\
 &                           & sum & 10.3307 &                         \\
\midrule
\multirow{4}{*}{300}

 & \multirow{2}{*}{Flagser} & max & 18.8565 & \multirow{2}{*}{18271.2} \\
 &                           & sum & 13.2354 &                         \\
 

 & \multirow{2}{*}{$p$-Rips} & max & 31.8323 & \multirow{2}{*}{236695.8} \\
 &                           & sum & 30.4488 &                         \\
\bottomrule
\end{tabular}
\caption{%
Times to compute the filtration and persistent homology in dimension $0$ and $1$, and the number of simplices (filtration size), using the $p$-Rips filtration and Flagser with sum ($p=1$) and maximum ($p=\infty$) for randomly generated weighted directed graphs. Reported values are averages over 5 runs with different random seeds. Directed graphs were generated using the Erd\"os--R\'enyi model with edge probability 0.2, edge weights sampled uniformly in $[0.1, 1.0]$, and vertex weights set to zero.}\label{tab:filtration_runtimes}
\end{table}

\FloatBarrier

\subsection{Omitted Proofs and Results}\label{sec:proofs_appendix}
This section contains proofs and results omitted from the main text. First, we prove \cref{lem:graph_distance_metric} showing metric-like properties of the graph distance.
\begin{cleverproof}{lem:graph_distance_metric}\label{proof:graph_distance_metric}
We prove one property at a time.
  \begin{enumerate}[wide, labelindent=0pt]
    \item First, note that $\delta_\atensor(G,G')=e$ if and only if $\vec{\delta}_\atensor(G,G')=\vec{\delta}_\atensor(G',G)=e$. Furthermore, this is equivalent to having both $\lambda_{G(v,v')}\left(G'(v,v')\right)=e$ and $\lambda_{G'(v,v')}\left(G(v,v')\right)=e$ for all $v,v'\in V$. Since $\lambda_t$ is the left adjoint of $\mu_t$, this is again equivalent to $G(v,v')\leq G'(v,v')$ and $G'(v,v')\leq G(v,v')$ since $\mu_t(e)=t$ for all $t\in L$. As $\leq$ is antisymmetric, the result follows.
    \item This follows from the fact that $\sup\{s,t\}=\sup\{t,s\}$. 
    \item We start by considering the directed case. By \cref{lem:graph_distance_is_a_bound}, we have the inequalities $G''(v,v')\leq G(v,v')\atensor\vec{\delta}_\atensor(G,G'')$ and $G'(v,v')\leq G''(v,v')\atensor\vec{\delta}_\atensor(G'',G')$ for all $v,v'\in V$ so by \ref{cond:l1}
    \begin{equation*}
    G(v,v')\atensor \vec{\delta}_\atensor(G,G'')\atensor \vec{\delta}_\atensor(G'',G')\geq G''(v,v')\atensor \vec{\delta}_\atensor(G'',G')\geq G'(v,v')\quad\text{for all }v,v'\in V. 
    \end{equation*}
    This means that $\vec{\delta}_\atensor(G,G'')\atensor \vec{\delta}_\atensor(G'',G')$ is an upper bound of the set $\{\lambda_{G(v,v')}\left(G'(v,v')\right)\}_{v,v'\in V}$, and since $\vec{\delta}_\atensor(G,G')$ is the least upper bound we get $\vec{\delta}_\atensor(G,G')\leq\vec{\delta}_\atensor(G,G'')\atensor\vec{\delta}_\atensor(G'',G')$.
    
    Now, we can extend this to the undirected case by noting
    \begin{align*}
    \delta_\atensor(G,G')&=\sup\left\{\vec{\delta}_\atensor(G,G'),\vec{\delta}_\atensor(G',G)\right\}\\&\leq\sup\left\{\vec{\delta}_\atensor(G,G'')\atensor\vec{\delta}_\atensor(G'',G'),\vec{\delta}_\atensor(G',G'')\atensor\vec{\delta}_\atensor(G'',G)\right\}\\
    &\leq\sup\left\{\vec{\delta}_\atensor(G,G''),\vec{\delta}_\atensor(G'',G)\right\}\atensor\sup\left\{\vec{\delta}_\atensor(G'',G'),\vec{\delta}_\atensor(G',G'')\right\}\\&=\delta_\atensor(G,G'')\atensor\delta_\atensor(G'',G')
    \end{align*}
    where the first inequality uses the non-decreasing property of $\atensor$, and the second inequality follows from eq.\ \eqref{eq:supminkowski} and commutativity of $\atensor$. 

    That $(\mathcal{G}(V,L), \delta_{+_p})$ is an extended metric follows directly, using $a+_p b\leq a+b$ from equation \eqref{eq:minkowski_for_p_norm_counting_measure}.
  \end{enumerate}
\end{cleverproof}

We will use the following lemma in the proof of \cref{lem:generalized_network_distance_is_psuedo_metric_like} when showing that the generalized network distance satisfies a triangle inequality.

\begin{lemma}\label{lem:adjoint_sup_triangle_inequality}
Let $L=(L,\leq,\atensor,e)$ be a symmetric monoidal lattice and let $G,G'\colon V\times V\to L$ be $L$-graphs. Then
\begin{equation*}
    \vec{\delta}_\atensor(G,G')\leq\sup_{v,v'\in V}\lambda_{G(v,v')}\left(u\right)\atensor\sup_{v,v'\in V}\lambda_u\left(G'(v,v')\right)
\end{equation*}
for all $u\in L$.
\end{lemma}
\begin{proof}
Let $u\in L$. From \cref{eq:left_adjoint_lambda_t} we have $\lambda_{G(v,v')}(u)=\inf\left\{c\in L\mid u\leq G(v,v')\atensor c\right\}$, and since $\lambda_t$ is left adjoint to $\mu_t$, it follows that 
\begin{equation*}
u\leq G(v_0,v'_0)\atensor\sup_{v,v'\in V}\lambda_{G(v,v')}(u).
\end{equation*}
for all $v_0,v'_0\in V$. A similar argument gives that
\begin{equation*}
G'(v_0,v'_0)\leq u\atensor\sup_{v,v'\in V}\lambda_u\left(G'(v,v')\right).    
\end{equation*}
for all $v_0,v'_0\in V$. By combining the above inequalities, we get that
\begin{equation*}
G'(v_0,v'_0)\leq G(v_0,v'_0)\atensor\left(\sup_{v,v'\in V}\lambda_{G(v,v')}(u)\atensor\sup_{v,v'\in V}\lambda_u\left(G'(v,v')\right)\right), 
\end{equation*}
and by applying the adjunction once more
\begin{equation*}
\lambda_{G(v_0,v'_0)}\left(G'(v_0,v'_0)\right)\leq\sup_{v,v'\in V}\lambda_{G(v,v')}(u)\atensor\sup_{v,v'\in V}\lambda_u\left(G'(v,v')\right)
\end{equation*}
for all $v_0,v'_0\in V$. Since $\vec{\delta}_\atensor(G,G')$ is the least such upper bound, the result follows.
\end{proof}

We now prove \cref{lem:generalized_network_distance_is_psuedo_metric_like} showing that the generalized network distance has pseudometric-like properties.

\begin{cleverproof}{lem:generalized_network_distance_is_psuedo_metric_like}\label{proof:generalized_network_distance_is_psuedo_metric_like}
We let $\supa\colon T^m\to T$ denote the map $(t_1,\ldots,t_m)\mapsto\sup_it_i$.
\begin{enumerate}[wide, labelindent=0pt]
\item Suppose $G_1=G_2\colon V\times V\to T^m$. Letting $C$ be the diagonal in $V\times V$ with projection maps $\pi_i$, we get that $\pi_1^*G_1=\pi_2^*G_2$. Thus, by \cref{lem:graph_distance_metric} $\delta_\atensor(\pi_1^*G_1,\pi_2^*G_2)=e$ so the generalized network distance is $e$ by \ref{cond:l2} and the fact that $\supa(e)=e$ in $T$.
\item Suppose $G_1$ and $G_2$ have vertex sets $V_1$ and $V_2$, respectively. The bijection $V_1\times V_2\to V_2\times V_1$ sending $(v_1,v_2)$ to $(v_2,v_1)$ induces a bijection between $\mathcal{C}(V_1,V_2)$ and $\mathcal{C}(V_2,V_1)$. Consequently, the result follows since $\delta_\atensor$ is symmetric by \cref{lem:graph_distance_metric}.
\item Let $G_i\colon V_i\times V_i\to T^m$ be $T^m$-graphs for $i=1,2,3$. Suppose we have correspondences $\widetilde{C}\in\mathcal{C}(V_1,V_3)$ and $\widehat{C}\in\mathcal{C}(V_3,V_2)$ with projection maps $\widetilde{\pi}_i$ and $\widehat{\pi}_i$, respectively. Define the composite correspondence
\begin{equation}\label{eq:composite_correspondence}
    C:=\widehat{C}\circ\widetilde{C}=\left\{(v_1,v_2)\in V_1\times V_2\mid \exists v_3\in V_3\text{ such that }(v_1,v_3)\in\widetilde{C}\text{ and }(v_3,v_2)\in\widehat{C}\right\}
\end{equation}
with projection maps $\pi_i$. To see that $C$ is in fact a correspondence is straightforward. Hence, we have that $\operatorname{d}_\mathcal{N}^\atensor(G_1,G_2)\leq \supa(\delta_\atensor(\pi_1^*G_1,\pi_2^*G_2))$. What remains to show is that
\begin{equation}\label{eq:graph_distance_inequality_correspondences}
\delta_\atensor(\pi_1^*G_1,\pi_2^*G_2)\leq\delta_\atensor(\widetilde{\pi}_1^*G_1,\widetilde{\pi}_2^*G_3)\atensor\delta_\atensor(\widehat{\pi}_1^*G_3,\widehat{\pi}_2^*G_2)
\end{equation}
as the result then follows from taking infimums over all correspondences since $s\leq t\atensor r$ in $T^m$ implies $\supa(s)\leq \supa(t)\atensor \supa(r)$ in $T$. By \cref{lem:adjoint_sup_triangle_inequality}, we have
\begin{equation*}
\vec{\delta}_\atensor(\pi_1^*G_1,\pi_2^*G_2)\leq\sup_{(v_1,v_2),(v'_1,v'_2)\in C}\lambda_{G_1(v_1,v'_1)}\left(u\right)\atensor\sup_{(v_1,v_2),(v'_1,v'_2)\in C}\lambda_u\left(G_2(v_2,v'_2)\right)    
\end{equation*}
for all $u$ in $T^m$. In particular, it holds for $u=G_3(v_3,v'_3)$ for all $v_3,v'_3\in V_3$ witnessing $(v_1,v_2),(v'_1,v'_2)\in C$ in the sense of \cref{eq:composite_correspondence}. Hence,
\begin{align*}
\vec{\delta}_\atensor(\pi_1^*G_1,\pi_2^*G_2)&\leq\qquad\sup_{\mathclap{\overset{\,}{(v_1,v_2),(v'_1,v'_2)\in C}}}\quad\lambda_{G_1(v_1,v'_1)}\left(G_3(v_3,v'_3)\right)\quad\atensor\quad\sup_{\mathclap{\overset{\,}{(v_1,v_2),(v'_1,v'_2)\in C}}}\quad\lambda_{G_3(v_3,v'_3)}\left(G_2(v_2,v'_2)\right)\\
&\leq\qquad\sup_{\mathclap{\overset{\,}{(v_1,v_3),(v'_1,v'_3)\in \widetilde{C}}}}\quad\lambda_{G_1(v_1,v'_1)}\left(G_3(v_3,v'_3)\right)\quad\atensor\quad\sup_{\mathclap{\overset{\,}{(v_3,v_2),(v'_3,v'_2)\in\widehat{C}}}}\quad\lambda_{G_3(v_3,v'_3)}\left(G_2(v_2,v'_2)\right)\\
&=\vec{\delta}_\atensor(\widetilde{\pi}_1^*G_1,\widetilde{\pi}_2^*G_3)\atensor\vec{\delta}_\atensor(\widehat{\pi}_1^*G_3,\widehat{\pi}_2^*G_2)
\end{align*}
By a similar argument to the one in the proof of \cref{lem:graph_distance_metric} we have shown that \cref{eq:graph_distance_inequality_correspondences} holds.
\end{enumerate}

That $\operatorname{d}^{+_p}_{\mathcal{N}}$  is an extended pseudo-metric on $\Gph([0,\infty]^m)$ follows directly, using $a+_p b\leq a+b$ from equation \eqref{eq:minkowski_for_p_norm_counting_measure}.
\end{cleverproof}

\subsection{Products of Symmetric Monoidal and Duoidal Lattices}\label{sec:productappendix}
Consider two symmetric monoidal lattices $(L_1,\leq_1,\tensor_1,e_1)$ and $(L_2,\leq_2,\tensor_2,e_2)$. For $i=1,2$, we write $\pi_i:L_1\times L_2\to L_i$ for the projection $\pi_i(l_1,l_2)=l_i$. We define the coordinatewise product $(L_1\times L_2,\leq,\tensor,(e_1,e_2))$ by
\begin{itemize}
    \item $l\leq l'$ whenever $\pi_i(l)\leq_i \pi_i(l')$ for $i=1,2$ and
    \item $l \tensor l'= (\pi_1(l)\tensor_1 \pi_1(l'), \pi_2(l)\tensor_2 \pi_2(l'))$.
\end{itemize}

Note that $\pi_i(l\tensor l')=\pi_i(l)\tensor_i\pi_i(l')$ for $i=1,2$, and $a=b$ if and only if $\pi_i(a)=\pi_i(b)$ for $i=1,2$. These facts help us to show the following:

\begin{proposition}
    The coordinatewise product $(L_1\times L_2,\leq,\tensor,(e_1,e_2))$ is a symmetric monoidal lattice. 
\end{proposition}
\begin{proof}
    \textit{Lattice:} Let $S\subseteq L_1\times L_2$. The point $(\inf \pi_1(S), \inf \pi_2(S))\leq s$ for all $s\in S$. Furthermore, if $d\leq s$ for all $s\in S$, then $\pi_i(d)\leq \pi_i(s)$ for all $\pi_i(s)\in\pi_i(S)$, thus $d\leq (\inf \pi_1(S), \inf \pi_2(S))$. So, $\inf S=(\inf \pi_1(S), \inf \pi_2(S))$, and similarly $\sup S=(\sup \pi_1(S), \sup \pi_2(S))$.

    \textit{Commutative monoid:} 
    For commutativity, $\pi_i(a\tensor b)= \pi_i(a)\tensor_i \pi_i(b)=\pi_i(b)\tensor_i \pi_i(a)=\pi_i(b\tensor a)$ for $i=1,2$. For associativity, we have 
    $
        \pi_i(a\tensor (b\tensor c))=\pi_i(a)\tensor_i \pi_i(b\tensor c)=\pi_i(a)\tensor_i (\pi_i(b)\tensor_i\pi_i(c))=(\pi_i(a)\tensor_i \pi_i(b))\tensor_i\pi_i(c)=\pi_i((a\tensor b)\tensor c)
    $
    for $i=1,2$. Finally, $(l_1,l_2)\tensor (e_1,e_2)=(l_1\tensor_1 e_1,l_2\tensor_2 e_2)=(l_1,l_2)$. 

    \textit{Symmetric monoidal lattice:} Starting with \ref{cond:l2}, we have $e=\inf (L_1\times L_2) = (\inf L_1, \inf L_2)=(e_1,e_2)$. For \ref{cond:l1}, let $a\leq a'$ and $b\leq b'$ in $(L_1\times L_2,\leq)$. In particular, for $i=1,2$, we have $\pi_i(a)\leq \pi_i(a')$ and $\pi_i(b)\leq \pi_i(b)$. By the \ref{cond:l1} property for $\tensor_i$ we have  $\pi_i(a\tensor b)=\pi_i(a) \tensor_i \pi_i(b)\leq \pi_i(a')\tensor_i \pi_i(b')=\pi_i(a'\tensor b')$ for $i=1,2$, so $a\tensor b\leq a'\tensor b'$.
\end{proof}

Now, if we have two symmetric duoidal lattices $(L_1,\leq_1,\tensor_1,\oplus_2,e_1)$ and $(L_2,\leq_2,\tensor_2,\oplus_2,e_2)$, we have the coordinatewise product $(L_1\times L_2,\leq,\tensor,\oplus,(e_1,e_2))$, where both $(L_1\times L_2,\leq,\tensor,(e_1,e_2))$ and $(L_1\times L_2,\leq,\oplus,(e_1,e_2))$ are defined coordinatewise.

\begin{proposition}
    The coordinatewise product $(L_1\times L_2,\leq,\tensor,\oplus,(e_1,e_2))$ is a symmetric duoidal lattice.
\end{proposition}
\begin{proof}
    We have yet to show the \ref{cond:bl2} property. Let $a,b,c,d\in L_1\times L_2$. The inequalities $(\pi_i(a)\atensor_i \pi_i(b))\tensor_i(\pi_i(c)\atensor_i \pi_i(d))\leq_i (\pi_i(a)\tensor_i \pi_i(c))\atensor_i(\pi_i(b)\tensor_i \pi_i(d))$ holds for $i=1,2$. Since $\pi_i$ commutes with both $\tensor$ and $\oplus$, we get $\pi_i((a\atensor b)\tensor (c\atensor d))\leq_i \pi_i((a\tensor c)\atensor(b\tensor d))$.
\end{proof}
\end{appendices}

\clearpage

\bibliographystyle{plainurl}
\bibliography{bibliography}

\end{document}